\documentclass[11pt]{article}

\usepackage[a4paper,margin=1in]{geometry}
\usepackage{amsfonts,amssymb,amsmath,amsthm}
\usepackage{graphicx}
\usepackage{psfrag}
\usepackage{calc}
\usepackage{epic}
\usepackage[english]{babel}
\usepackage{hyperref}
\usepackage{float}
\usepackage{latexsym}
\usepackage{color}
\usepackage{subfigure,dsfont}
\allowdisplaybreaks

\DeclareMathOperator*{\argmax}{arg\,max}

\usepackage{hyperref}
\newcommand{\footremember}[2]{%
    \footnote{#2}
    \newcounter{#1}
    \setcounter{#1}{\value{footnote}}%
}
 
\title{Stability and moment bounds under utility-maximising service allocations: finite and infinite networks
}
\author{%
  Seva Shneer\footremember{HW}{Heriot-Watt University, V.Shneer@hw.ac.uk}%
  \and Alexander Stolyar\footremember{UIUC}{University of Illinois at Urbana-Champaign, stolyar@illinois.edu}%
  }
\date{}
\newtheorem{theorem}{Theorem}
\newtheorem{lemma}[theorem]{Lemma}

\newtheorem{corollary}[theorem]{Corollary}

\theoremstyle{remark}
\newtheorem{remark}{Remark}

\newcommand{\be}{ \begin{equation}}
\newcommand{\ee}{\end{equation}}
\newcommand{\ben}{ \begin{equation*}}
\newcommand{\een}{\end{equation*}}

\newcommand{\beql}[1]{\begin{equation}\label{#1}}
\newcommand{\eeql}{\end{equation}}
\newcommand{\eqn}[1]{(\ref{#1})}

\def\E{{\mathbb E}}
\def\P{{\mathbb P}}
\def\I{{\mathbb I}}
\def\N{{\mathcal N}}
\def\C{{\mathcal C}}

\def\ox{{\bar{x}}}
\def\ol{{\bar{\lambda}}}
\def\on{{\bar{\nu}}}
\def\om{{\bar{\mu}}}

\begin{document}
\maketitle

\begin{abstract}
We study networks of interacting queues governed by utility-maximising service-rate allocations in both discrete and continuous time. For {\em finite} networks we establish stability and some steady-state moment bounds under natural conditions and rather weak assumptions on utility functions. These results are obtained using direct applications of Lyapunov-Foster-type criteria, and apply to a wide class of systems, including those for which fluid limit-based approaches are not applicable.

We then establish stability and some steady-state moment bounds for two classes of {\em infinite} networks, with single-hop and multi-hop message routes. These results are proved by considering the infinite systems as limits of their truncated finite versions. The uniform moment bounds for the finite networks play a key role in these limit transitions.
\end{abstract}



\section{Introduction}

In this paper we consider networks of interacting queues. These models are primarily motivated by wireless systems, where the interference between simultaneous transmissions by different nodes imposes certain constraints. For example, ``neighbouring'' nodes may not be allowed to transmit simultaneously and/or a node's effective transmission rate depends on the transmission powers of the node and its neighbours. However, the basic model studied in this paper takes a more abstract point of view, namely, it is an arbitrary network of queues such that individual instantaneous service rates may depend on the state of the entire system. 
The network state is represented as a set $\bar X = (X_i, ~i\in \mathcal N)$ of the queue lengths $X_i$ 
(of jobs, or messages, or customers) at the network nodes $i\in \mathcal N$. 
Each node receives exogenous arrivals of jobs (messages).
We consider both {\em discrete-time} and {\em continuous-time} models,
and both {\em finite} and {\em infinite} networks. Also, in addition to {\em single-hop} networks, where each job leaves the system after its service is completed, we consider one special class of {\em multi-hop} networks, where, after a service completion at one node, a customer may leave the network or be routed to another node, and these routing decisions are taken according to a certain random procedure.

In discrete-time models time is divided into {\em slots} of the same (unit) size, and each job (or message) takes exactly one slot to complete  service at a node. A service allocation algorithm (or rule) is any mapping (deterministic or random) of a network state $\bar X$ into the set 
of nodes that serve jobs (transmit messages) in a slot. (See, e.g., \cite{SnSt2018} for a recent model of an algorithm in discrete time employing a random procedure.) The instantaneous service rate $\mu_i$ of node $i$ in a slot is the probability that it will serve a job. Thus, the deterministic mapping $\bar \psi(\bar X) = (\psi_i(\bar X), ~i\in \mathcal N)$ of a network state $\bar X$ into a set of instantaneous service rates $\bar \mu = (\mu_i, ~i\in \mathcal N)$ is determined by the service algorithm; the mapping $\bar \psi(\bar X)$ is referred to as a service rate allocation algorithm (or rule). 

In continuous-time models, the instantaneous service rate $\mu_i$ of a node represents the intensity of the Poisson process modelling departures (service completions) of the node. In this case, the service rate allocation algorithm $\bar \psi(\bar X)$, mapping a network state $\bar X$ into a set of instantaneous service rates $\bar \mu$, is all that is needed to specify the service allocation algorithm 
(see, e.g. \cite{Baccelli2018} for a recent model of an algorithm in continuous time).

In this paper we study service allocation algorithms (in both discrete and continuous time), such that the corresponding
service rate allocation $\bar \psi(\bar X)$ maximises some utility function within some set $\mathcal C$. In some cases, the set $\mathcal C$
arises naturally as the set of all feasible instantaneous rates $\bar \mu$ {\em given the model structure}, but not necessarily. For instance, in the networks considered in \cite{SnSt2018}, as well as Section \ref{sec:infinite_single} here, the set $\mathcal{C}$ is a subset of the set of all feasible rates. Our main goal is to obtain network {\em stability} conditions, in terms of set $\mathcal C$. For example, our main stability results (Theorems \ref{thm:fairness_stability} and \ref{thm:stability_cont}) for {\em single-hop} networks show that the network is stable when the exogenous arrival rates $\bar \lambda = (\lambda_i, ~i\in \mathcal N)$ are (``strictly'') within set $\mathcal C$. In addition to stability, we are able to obtain some steady-state moments bounds (Theorems \ref{thm:fair_tight} and \ref{thm:fair_tight_cont}). These moment bounds turn out to be key to establishing stability of infinite networks, because they allow a limit transition from finite to infinite networks (Theorems \ref{thm:periodic} and \ref{thm:periodic_cont}).

Service rate allocations $\bar \psi(\bar{X})$, under many natural service allocation algorithms, are such that
 $\psi_i(\bar{X})$ is decreasing in each $X_j$ for $j \neq i$, as for these algorithms a higher load in queue $j$ usually leads to all other queues receiving less service. This property is in fact satisfied by the rates defined by algorithms introduced in \cite{Baccelli2018} and \cite{SnSt2018} that we will study here as examples. We would like to emphasise, however, that for our general results we are {\em not} going to make this assumption. Our motivation for this stems, again, from wireless networks where there are many competing factors at play and in many situations $\psi_i$ may not be decreasing in $X_j$ for some $j \neq i$ (see, e.g. the model considered in \cite{SnSt2017}, where the authors consider an algorithm designed to ensure avoidance of conflicts which gives advantage to a transmitter if its non-immediate neighbours are transmitting). This leads to a potentially wide range of possible assumptions on the dependence of service rates assigned to different queues on the state of the network.

We are interested in conditions guaranteeing stability. In finite networks  stability, informally speaking, means the ability of all queues to complete service of all jobs, without the number of outstanding jobs building up infinitely. More formally, this means that the Markov chain $\bar X(\cdot)$ is positive recurrent. This also implies the existence and uniqueness of a stationary distribution. 

In infinite networks, by stability we will understand the existence of a proper invariant distribution. In the cases when the system process is {\em monotone}, this implies that the process distribution converges to a proper steady-state (namely, the lower invariant measure), starting from the ``empty'' initial state, as time goes to infinity.

An important concept, explored extensively in the literature, is that of {\em maximum stability} (or {\em throughput optimality}).
To illustrate this concept, consider a finite network and let $\C$ be the set of all feasible {\em long-term} rates that can be provided to the nodes, given model constraints. Such a set $\C$ is typically convex.
Then, an algorithm is called maximally stable (or, throughput-optimal) if it guarantees stability as long as
$\ol < \on$ for some $\on \in \C$; in other words, essentially, as long as the stability is feasible at all. For a large class of networks,
the celebrated MaxWeight algorithm (\cite{Tassiulas1992}) and $\alpha$-fair algorithm are known to be maximally stable.
(See \cite{Kelly1998, Mo2000, Roberts2000} for introduction of the fair-allocation concepts and \cite{Bonald2001, DeVeciana2001} for stability proofs.) These algorithms, however, are {\it centralised} in that service-rate allocations are given by a solution to an optimisation problem that needs to be found by a certain central entity. There are also {\it decentralised} algorithms (where each node regulates its own behaviour according to its queue length) guaranteeing maximal stability (see \cite{JW2010, Shah2012}), but they are known to suffer from large job delays. (This, in particular, prompted the introduction and analysis of algorithms, which are not maximally stable, and instead ensure stability
for $\ol$ within a ``smaller'' set than the set of all feasible long-term rates, and this stability being not necessarily convex. See, e.g., \cite{St2005alq}.)

Some maximally stable algorithms are designed in such a way that the average service rates maximise a certain utility function. A notable example is presented by $\alpha$-fair algorithms where the rates $\psi_i$ are such that
$$
\overline{\psi} \in \argmax_{\overline{\mu} \in \C} \sum_i X_i \frac{1}{1-\alpha} \left(\frac{\mu_i}{X_i}\right)^{1-\alpha}, ~~\mbox{when}~~\alpha > 0, ~\alpha \ne 1,
$$
or
$$
\overline{\psi} \in \argmax_{\overline{\mu} \in \C} \sum_i X_i \log (\mu_i/X_i), ~~\mbox{when}~~\alpha = 1,
$$
where the set $\C$ is usually assumed to be convex.
 The known stability proofs are based on the fluid-limit approach (\cite{RybSt92, Dai95, St95}) and, in particular, implicitly use the fact that $\alpha$-fair service-rate allocations are $0$-homogeneous (or, asymptotically $0$-homogeneous), which allows a relatively simple characterisation of fluid-limit dynamics. 
 
In this paper we consider general utility-optimising algorithms, which, in particular, do not necessarily assign $0$-homogeneous rates to queues. We also do {\em not} require that the maximisation set is necessarily convex.
Our goal is threefold. First, we show that these very general algorithms for finite networks ensure stability when $\ol$ is within $\C$. Second, we also find some moment bounds for the stationary queue-length distributions. And finally, we demonstrate how our moment bounds may be used to extend the stability results and moment bounds to some infinite networks.

In the first part of our paper, we consider a class of general utility-optimising algorithms and prove that they are stable when $\ol$ is within $\C$. Namely, we study average service-rate allocations $\psi_i$ such that
$$
\overline{\psi} \in \argmax_{\overline{\mu} \in \C} \sum_i g(X_i) h(\mu_i),
$$
with some conditions on the functions $g$ and $h$. Our conditions do not imply that the service-rate allocations are $0$-homogeneous, hence the existing stability results, based on fluid limits, do not apply. Moreover, we do not even require that the function $g$ is defined for non-integer values of the argument. Our results are valid for a large class of functions $g$ such that $g(n+1)/g(n) \to 1$ as $n \to \infty$. This class includes functions $g(n) = n^{\alpha}$ used in $\alpha$-fair allocations, as well as functions of the form $g(n) = e^{\log^\beta n}$ with $\beta > 0$ and $g(n) = e^{n^{\gamma}}$ with $\gamma \in (0,1)$, among others. Our results are also valid for a very general class of functions $h$. Our stability proofs in both discrete- and continuous-time settings are based on the direct application of the Lyapunov-Foster techniques.

In discrete time, for our general results (Theorem \ref{thm:fairness_stability} and Theorem \ref{thm:fair_tight}) we impose a strong additional assumption that the number of arrivals into each queue in a time slot is given by a Bernoulli random variable, while if we restrict our attention to a particular scenario of interest (see Theorem \ref{thm:bound_third}), we only need to assume a finite third moment of the per-slot number of arrivals.
 In continuous time however (which is the standard setting for $\alpha$-fair allocations) no additional assumptions are needed. We note again that we do not assume that the set $\C$ is convex.

Once stability is established, one is interested in characteristics of the stationary regime. For both discrete- and continuous-time settings, we demonstrate how essentially the same techniques used to prove stability may be employed to establish explicit bounds on the moments of queue states in stationarity.

These bounds are interesting in their own right, especially as very few results are known on the stationary regimes of networks governed by utility-maximising algorithms. We note \cite{Shah2014} where an exponential bound has been established for the tail of the total stationary queue length of a system under an $\alpha$-fair algorithm in a Markovian setting, and \cite{Meyn1995} where sufficient conditions for the existence of finite moments were established for general arrival streams. We note however that the results of both \cite{Meyn1995} and \cite{Shah2014} imply finiteness of some moments of the stationary queue-length distributions but do not imply any bounds on them as the various constants are not explicit. In the second part of our paper, having explicit bounds is crucial for the analysis of some infinite networks.

In the second part of the paper, we apply the moment bounds established in this paper to obtain stability results for infinite networks in discrete and continuous time considered in recent papers \cite{SnSt2018} and \cite{Baccelli2018}, respectively. The models considered in the two papers are motivated by different wireless networks but share similar service-rate allocations. As our stability and moment analysis is based on service-rate allocations only, it allows us to handle both discrete and continuous cases, and particular characteristics of the two models (which are very different), beyond the service-rate allocation, do not play any role in the proofs.

The simplest example of the two networks (results for more general settings are presented in the paper; we focus on a simple example in the introduction only) is given by nodes located on an infinite line $\mathbb{Z}$ and such that, given the state of the system $\bar X$, the service-rate allocation is given by
$$
\psi_i = \frac{X_i}{X_{i-1}+X_i+X_{i+1}}.
$$
The so-called {\em rate stability} (guaranteeing the queue lengths do not grow linearly in time) is demonstrated in both discrete and continuous settings in \cite{SnSt2018} for arrival rates $\ol$ within some natural set $\C$. Authors of \cite{Baccelli2018} considered a continuous-time model where arrival rates into all nodes are the same and equal to $\lambda$, say. They consider  
the system dynamics on intervals $(-n,\ldots,n)$ viewed as a circle, with a growing $n$. These systems are stable for any $n$, provided $\lambda < 1/3$, and one can thus consider their stationary measures. Using the natural monotonicity of the corresponding process, and tightness of these measures, a stationary measure (in fact, the lower invariant measure) is constructed for the infinite network. To establish uniqueness of this stationary measure among those with finite second moments of the queue lengths, one needs a bound on the second moments of stationary measures of the systems on the circle, independent of their size. This was not established in \cite{Baccelli2018} and left as a conjecture (Conjecture 1.12).

Our analysis is based on showing that the rates of \cite{Baccelli2018} and \cite{SnSt2018} are in fact utility-maximising (or $2$-fair in the $\alpha$-fair terminology) in a certain natural set $\C$ -- a fact already mentioned in \cite{SnSt2018}. This allows us to use our results on stability and moment bounds for finite systems. In particular, our moment bounds immediately imply a uniform (not depending on the size of the network) bound for second moments. This, in turn, proves \cite[Conjecture 1.12]{Baccelli2018} in the case of identical arrival rates, 
with all its implications, including the uniqueness of the stationary measure constructed there, among stationary measures with finite second moments of the queue lengths.

Our analysis, however, allows to demonstrate the existence of a stationary measure with a finite second moment in far more general settings where arrival rates do not need to be the same at all nodes, but may be periodic (or dominated by periodic). Our analysis of systems with the specific service rates of \cite{Baccelli2018} and \cite{SnSt2018} does rely on the existence and stability of fluid limits of the systems.

As our analysis is based on utility maximisation and continuity properties of the processes (see Section \ref{sec:basic} for the definition of continuity property), it is not specific to the rates considered in \cite{Baccelli2018} and \cite{SnSt2018} and may be applied to other infinite networks. We present further examples of models where the same strategy applies. The examples provided in this paper are however not exhaustive.

Finally, we also consider a multi-hop network of \cite{SnSt2018}. In a multi-hop network, jobs, after being served at one queue, may leave the network or join another queue to be served there. The analysis of multi-hop networks is notoriously difficult. As in \cite{SnSt2018}, we restrict our attention to symmetric routing. We use similar techniques to the ones we applied in the single-hop setting to first obtain moment bounds for finite networks and then apply these bounds to establish stability of an infinite network. Stability in this case is weaker than that obtained in the single-hop case as the multi-hop network lacks monotonicity, which is at the core of the construction of the lower invariant measure in \cite{Baccelli2018}.

To summarise, our contributions are the following:

\begin{itemize}

\item We provide a proof of stability of utility-maximising algorithms in a general setting, covering cases in which fluid limit technique cannot be applied. In particular, we do not assume that service rate allocations are $0$-homogeneous, and thus use more general utility functions compared to the classical $\alpha$-fair algorithms. This comes at the expense of additional assumptions on the arrival processes in discrete time. There are, however, no additional assumptions made in the case of a Markovian (driven by Poisson arrivals and departures) continuous-time system.

\item Using a similar approach, we obtain steady-state moment bounds, provided stability conditions hold.

\item The same ideas allow us to obtain further explicit steady-state moment bounds in some special cases of interest, which, in turn, allow us to establish stability and moment bounds of some infinite networks. When restricting our attention to some specific service rates, we do use the existence and stability of fluid limits.

\item We use similar techniques to establish moment bounds for a certain finite multi-hop network and use these bounds to establish stability and moment bounds for its infinite version.

\end{itemize}

\subsection{Structure of the paper}

In order to facilitate a smooth presentation, we first present our main results in the discrete-time setting. Section \ref{sec:finite} is devoted to finite networks. We describe the model, present our stability results and steady-state moment bounds. Sections \ref{sec:infinite_single} and \ref{sec:infinite_multi} are devoted to infinite networks, single- and multi-hop settings, respectively.  In Section~\ref{sec:infinite_single}, we first analyse the model introduced in \cite{Baccelli2018} and \cite{SnSt2018} and later demonstrate how our analysis may be extended to treat other related models. Section \ref{sec:infinite_multi} is devoted to a symmetrical multi-hop extension of the model of \cite{Baccelli2018} and \cite{SnSt2018}.

Section \ref{sec:continuous} is devoted to the continuous-time setting. We describe the model, explain why the analysis is a simplified version of our analysis in the discrete-time setting and present continuous-time versions of our results.

The paper contains many parts, some of which are connected directly through statements while others are only connected through ideas. The structure of the paper is thus non-linear and we present an illustration of how the various parts are related in figure \ref{fig:connections}.

\begin{figure}
\centering
\includegraphics[width=\linewidth]{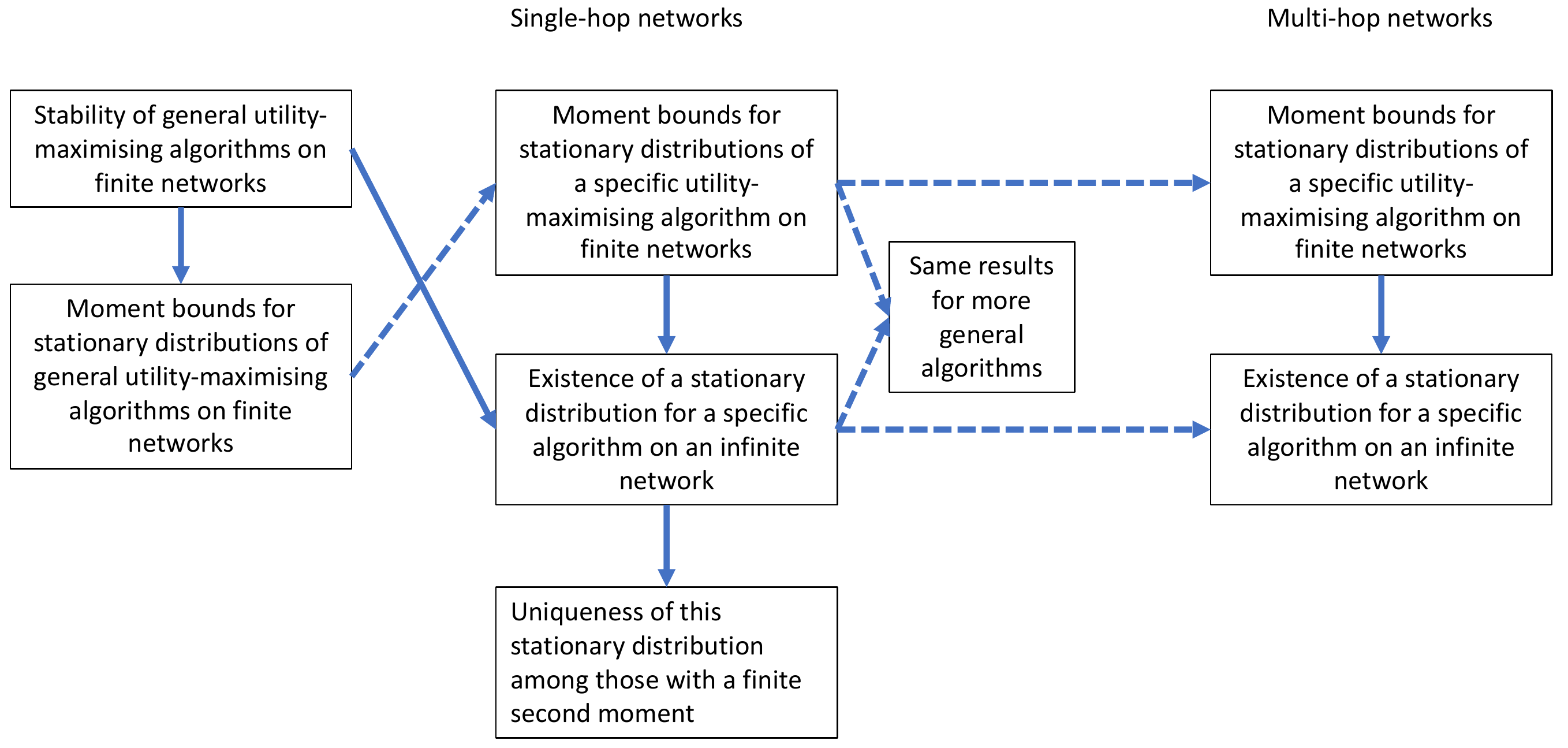}
\caption{An illustration of connections between various parts of the paper. Solid lines represent strict logical connections, dashed lines - connections in terms of ideas}
\label{fig:connections}
\end{figure}

\subsection{ Basic notation, conventions and definitions} \label{sec:basic}

We will use the following notation throughout: $\mathbb{R}$ and $\mathbb{R}_+$ are the sets of real and real non-negative numbers, respectively; $\mathbb{Z}^d$ is the $d$-dimensional lattice; $\mathbb{Z}_{+}$ is the set of non-negative integers;
$\bar{y}$ means (finite- or infinite-dimensional) vector $(y_i)$; for a finite-dimensional vector $\bar{y}$, $\|y\| = \sum_i |y_i|$;
for a set of functions $(f_i)$ and a vector $(y_i)$, $\bar{f}(\bar{y})$ denotes the vector $(f_i(\bar{y}))$; vector inequalities are understood component-wise; we also use the convention that $0/0 = 0$. 

Abbreviation {\em w.p.1} means {\em with probability $1$}.
The convergence in distribution of random elements is denoted by $\Rightarrow$. 
A discrete-time random process $(Y(k), ~k=0,1,2, \ldots)$ is often referred to as $Y(\cdot)$, and similarly for a continuous-time process $(Y(t), ~t\ge 0)$.

We will say that a sequence of random processes $Y^{(m)}(\cdot), m=1,2, \ldots,$ and a random process $Y(\cdot)$ satisfy a {\em continuity} property, if the following holds. For any (random) initial state $Y(0)$, and any sequence of (random) initial states $Y^{(m)}(0), m=1,2, \ldots,$
such that $Y^{(m)}(0) \Rightarrow Y(0)$, all processes can be coupled (constructed on a common probability space) so that
$Y^{(m)}(k) \to Y(k)$ w.p.1, for any $k=0,1,\ldots$ (or $Y^{(m)}(t) \to Y(t)$ w.p.1, for any $t\ge 0$, for continuous time). This continuity property could be called generalised Feller-continuity, because in the special case when all $Y^{(m)}(\cdot)$ are copies of the same process $Y(\cdot)$, differing only by the initial state, the property defined above is Feller-continuity of $Y(\cdot)$; we call it continuity for short. 

\section{Finite single-hop networks: stability analysis and moment bounds} \label{sec:finite}

In this Section we consider finite single-hop networks, where a job, after being served at any queue (node), leaves the system. Since the number of nodes is finite, the process describing system evolution is a countable (irreducible) Markov chain. The finite-network process {\em stability} is defined as positive recurrence of the Markov chain, which (due to irreducibility) is equivalent to the existence of a unique stationary distribution. First, we introduce the model and make general assumptions, then state and prove stability results and finally obtain moment bounds on stationary distributions.

\subsection{Model and assumptions.} 
\label{sec:model_discrete_finite}

Assume that there are $N$ queues, each having its own arrival stream of jobs, and having an infinite buffer to store outstanding jobs. For models in discrete time, we will assume that all jobs require service that lasts $1$ time unit, time is split into {\em slots} of length $1$, and all arrivals and all service initiations happen at the beginning of a time slot, so that all services are completed by the end of a time slot. These assumptions are motivated mainly by wireless networks.

For convenience we assume that at the beginning of each time slot, first new services are started, and then new arrivals happen. We will denote time slots by $k=0,1,\ldots$. We can then write the evolution of the queue of node $i$ as
\begin{equation} \label{eq:dynamics_discrete}
X_i(k+1) = X_i(k) + \xi_i(k) - \eta_i(k), 
\end{equation}
where $\xi_i(k)$ denotes the number of new job arrivals into queue $i$ at time slot $k$, and $\eta_i(k)$ denotes the number of service completions in queue $i$ at time $k$.

We will assume that for each $i$, the sequence $\xi_i(k), ~k=0,1,2,\ldots$ consists of i.i.d. random variables such that $\E(\xi_i) = \lambda_i$, where, here and throughout, by $\xi_i$ we denote a random variable with the distribution of any of $\xi_i(k)$. Note that arbitrary dependence between random variables with different values of $i$ is allowed.

We will assume also that random variables $\eta_i(k)$ take values $0$ and $1$ and are such that, on average, they maximise a global utility function in the following sense. Denote
$$
\psi_i(\bar{x}) = \E(\eta_i(k)|\bar{X}(k) = \bar{x})
$$
and assume that
$\psi_i(\ox) \in [0,1]$ for all $\ox$, $\psi_i(\ox) = 0$ if $x_i=0$, and 
\begin{equation} \label{eq:fainess_def_1}
\bar{\psi}(\bar{x}) \in \argmax_{\bar{\mu} \in \C} \sum_i g(x_i) h(\mu_i),
\end{equation}
where the set $\mathcal C$ is compact and coordinate-convex (i.e. if a vector $\bar\mu$ belongs to $\C$ and $\bar\mu^* \le \bar\mu$ coordinate-wise, then $\bar \mu^* \in \C$). We impose, in addition,

{\it Condition (H):} the function $h: [0,\infty) \to \mathbb{R}$ is strictly increasing, differentiable and concave (both the cases $\lim_{y\downarrow 0} h(y) = h(0) > -\infty$ and $\lim_{y\downarrow 0} h(y) = -\infty$ are allowed);

and

{\it Condition (G):} the function $g: \mathbb{Z}_+ \to [0,\infty)$ is strictly increasing and such that
\begin{equation} \label{eq:cond_Delta}
\frac{g(y)}{\Delta(y)} \to \infty
\end{equation}
as $y \to \infty$, where $\Delta(y) = g(y+1) - g(y)$. Note that condition \eqref{eq:cond_Delta} is equivalent to
\begin{equation} \label{eq:cond_Delta_cor}
\frac{g(y+1)}{g(y)} \to 1,
\end{equation}
as $y \to \infty$.


\begin{remark} Note that for what is usually referred to as $\alpha$-fair algorithms, $g(y) = y^{\alpha}$ and $h(y) = \frac{y^{1-\alpha}}{1-\alpha}$ with $\alpha > 0$, $\alpha \neq 1$, or $g(y) = y$ and $h(y)=\log y$ , so all the above conditions hold.

The class of functions satisfying the conditions above is however much wider. It includes, for instance functions $g(y) = e^{\log^\beta y}$ with $\beta > 0$ and $g(y) = e^{y^\gamma}$ with~$0<\gamma<1$.
\end{remark}

Throughout the section, we are going to assume that
\beql{eq-sub-critical}
\mbox{There exists $\on \in \mathcal C$ such that $\ol < \on$.}
\end{equation}

We will also denote
\begin{equation} \label{eq:def_G}
G(z) = \sum_{y=0}^z g(y),
\end{equation}
and
\begin{equation} \label{eq:def_F}
F(\bar{y}) = \sum_i h'(\nu_i) G(y_i).
\end{equation}

\subsection{Stability.} \label{sec:stability_discrete_finite}

In this section we prove that the utility-maximising algorithms described in the previous section are stable as long as the average arrivals $\ol$ are within the set $\C$. Our proof does not use fluid limits which have been the standard tool for proving stability of algorithms of this type. The advantages and disadvantages of our approach are described in the Introduction.

\begin{theorem} \label{thm:fairness_stability}
Consider the discrete-time model in Section~\ref{sec:model_discrete_finite} and
assume that $\xi_i$ is a Bernoulli random variable with $\E(\xi_i) = \lambda_i$. Assume that the vector $\ol$ is such that condition \eqn{eq-sub-critical} holds. Then the Markov chain $\{\bar{X}(k), ~k=0,1,\ldots\}$ is positive recurrent.
\end{theorem}

\begin{proof}[Proof of Theorem \ref{thm:fairness_stability}.]

We will use the standard Lyapunov-Foster criterion \cite{Foster1953}. Fix $\varepsilon > 0$ such that $\lambda_i < \nu_i - \varepsilon$ for all $i$. Note that, due to \eqref{eq:fainess_def_1} and the concavity of the function $h$,
\begin{equation} \label{eq:fairness_inequality}
0 \le \sum_i g(x_i)(h(\psi_i(\bar{x})) - h(\nu_i)) \le \sum_i g(x_i)h'(\nu_i)(\psi_i(\bar{x})-\nu_i).
\end{equation}
We are going to consider
$$
\E\left(F(\bar{X}(1)) - F(\bar{X}(0))|\bar{X}(0) = \bar{x}\right) = \E\left(F(\bar{x}+\bar{\xi}(0) - \bar{\eta}(0))|\bar{X}(0) = \bar{x}\right) - F(\bar{x}).
$$
In what follows we are going to assume that $\bar{X}(0) = \bar{x}$ is fixed and will drop the dependence on this event. We will also write $\xi_i$ and $\eta_i$ instead of $\xi_i(0)$ and $\eta_i(0)$, for simplicity. We can write
\begin{align}
& \E(F(\bar{x} + \bar{\xi} - \bar{\eta})) - F(\bar{x}) = \E\left(\sum_i G(x_i+\xi_i-\eta_i) h'(\nu_i)\right) - \sum_i G(x_i)h'(\nu_i) \notag
\\ & =\sum_i h'(\nu_i)\biggl(\lambda_i \psi_i(\ox) G(x_i) + (1-\lambda_i)(1-\psi_i(\ox))G(x_i) \notag
\\ & + \lambda_i (1-\psi_i(\ox))G(x_i+1) + (1-\lambda_i)\psi_i(\ox) G(x_i-1) - G(x_i)\biggr) \notag
\\ & = \sum_i h'(\nu_i)\biggl(\lambda_i (1-\psi_i(\ox))(G(x_i+1) - G(x_i)) + (1-\lambda_i) \psi_i(\ox)(G(x_i-1)-G(x_i))\biggr) \notag
\\ & = - \sum_i h'(\nu_i) \lambda_i \psi_i(\ox)(G(x_i+1)+G(x_i-1)-2G(x_i)) \notag
\\ & + \sum_i h'(\nu_i)\bigl(\lambda_i(G(x_i+1)-G(x_i)) + \psi_i(\ox)(G(x_i-1)-G(x_i))\bigr) \notag
\\ & =  - \sum_i h'(\nu_i) \lambda_i \psi_i(\ox)(g(x_i+1) - g(x_i)) + \sum_i h'(\nu_i)\bigl(\lambda_i g(x_i+1) - \psi_i(\ox)g(x_i)\bigr) \notag
\\ & \le \sum_i h'(\nu_i)\bigl(\lambda_i g(x_i+1) - \psi_i(\ox)g(x_i)\bigr) \label{eq:discrete_cont}
\\ & = \sum_i h'(\nu_i) (\lambda_i - \psi_i(\ox)) g(x_i)  +  \sum_i h'(\nu_i)\lambda_i(g(x_i+1)-g(x_i)) \notag
\\ & = \sum_i h'(\nu_i) (\lambda_i - \nu_i) g(x_i) + \sum_i h'(\nu_i) (\nu_i - \psi_i(\ox)) g(x_i) + \sum_i h'(\nu_i)\lambda_i \Delta(x_i) \notag
\\ & \le - \varepsilon \sum_i h'(\nu_i) g(x_i) + \sum_i h'(\nu_i) \Delta(x_i) = \sum_i h'(\nu_i) (-\varepsilon g(x_i) + \Delta(x_i)), \notag
\end{align}
where in the last inequality we used \eqref{eq:fairness_inequality}.

This, together with \eqref{eq:cond_Delta}, implies that if $F(\bar{x})$ is large, then its expected drift may be made arbitrarily small, and certainly negative and bounded away from zero (and then the classical Lyapunov-Foster stability criterion in \cite{Foster1953} applies). 

Indeed, if $F(\bar{x})$ is larger than a certain constant, $C$ say, then there exists $i$ such that $h'(\nu_i) G(x_i) > C/N$, which, due to the fact that $G$ is strictly increasing, implies that $x_i > C_1$ for a certain constant $C_1$.

Condition \eqref{eq:cond_Delta} implies that there exists a constant $C_2$ such that
$$
\frac{g(y)}{\Delta(y)} \ge C_2
$$
for all $y$. If $\varepsilon C_2 > 1$, then the drift of $F$ is always negative. Assume now $\varepsilon C_2 < 1$. For any $C_3$ there exists $Y$ such that
$$
\frac{g(y)}{\Delta(y)} \ge C_3
$$
for all $y > Y$. We can always choose $C_3$ and $C$ so that $C_3 > 1/\varepsilon$ and $C_1 > Y$, and then the drift of $F$ may be bounded from above by
\begin{align*}
& \sum_i h'(\nu_i) g(x_i) \left(-\varepsilon + \frac{1}{C_2} \I(x_i \le Y) + \frac{1}{C_3} \I(x_i > Y)\right) \\ & \le (N-1) h_u g(Y) \left(-\varepsilon + \frac{1}{C_2}\right) + h_l g(C_1) \left(-\varepsilon + \frac{1}{C_3}\right),
\end{align*}
where $h_u = \max_i h'(\nu_i)$ and $h_l = \min_i h'(\nu_i)$. It is clear that we can always choose $C_1$ such that the above is negative.
\end{proof}

\subsection{Moment bounds.} 
\label{sec:bounds_discrete}

Once stability is established, one can employ arguments similar to those used in the proof of Theorem \ref{thm:fairness_stability} to obtain bounds on some moments of the stationary distributions of queue states. The stationary regime exists under conditions of Theorem \ref{thm:fairness_stability}; in this section we will write $\bar X$ to represent a random vector with the distribution equal to that of $\bar X(k)$ in the stationary regime.

\begin{theorem} \label{thm:fair_tight}
Assume that all conditions of Theorem \ref{thm:fairness_stability} hold and fix $\varepsilon > 0$ such that $\lambda_i < \nu_i - \varepsilon$ for each $i$. Then
$$
\sum_i h'(\nu_i) \E g(X_i) \le \frac{1}{\varepsilon} \sum_i h'(\nu_i) \E \Delta(X_i).
$$

\end{theorem}

\begin{proof}[Proof of Theorem \ref{thm:fair_tight}.]


Consider the process with an arbitrary fixed initial state $\bar X(0)$. 
Then, due to the assumptions on the input flows, $\E(F(\bar{X}(k))) < \infty$ for any $k\ge 0$.
By Theorem \ref{thm:fairness_stability} the process is stable, and therefore $\bar X(k)$ converges in distribution to $\bar X$. 
Following the lines of \eqref{eq:discrete_cont}, we have the following drift estimate:
\begin{align}
& \E \bigl(F(\bar{X}(k+1)) - F(\bar{X}(k)) ~|~\bar{X}(k) = \bar{x}\bigr) \notag
\\ & \le  \sum_i h'(\nu_i)\bigl(\lambda_i g(x_i+1) - \psi_i(\ox) g(x_i)\bigr) \label{eq:bound_discrete_cont}
\\ & \le - \varepsilon \sum_i h'(\nu_i) g(x_i) + \sum_i h'(\nu_i) \Delta(x_i) \le c, \notag
\end{align}
where $c$ is a fixed finite constant, and the last inequality follows from the properties of function $g$. Indeed, thanks to \eqref{eq:cond_Delta}, there exists a finite $Y$ such that
$$
\frac{g(y)}{\Delta(y)} \ge 1/\varepsilon
$$
for all $y \ge Y$. Denote by
$$
C = \min_{y \le Y} \frac{g(y)}{\Delta(y)}.
$$
Then
$$
\E \bigl(F(\bar{X}(k+1)) - F(\bar{X}(k)) ~|~\bar{X}(k) = \bar{x}\bigr) \le \left(-\varepsilon + \frac{1}{C}\right) \sum_i h'(\nu_i)g(x_i) \I(x_i < Y).
$$
Then $c$ may be taken to be
$$
c= 
\begin{cases}
0, \quad C \ge 1/\varepsilon, \\
N \left(-\varepsilon + \frac{1}{C}\right) g(Y) \max_i h'(\nu_i), \quad C < 1/\varepsilon
\end{cases}
$$

Then, we obtain
$$
\E[F(\bar{X}(k+1))-F(\bar{X}(k))] \le \E\left[- \varepsilon \sum_i h'(\nu_i) g(X_i(k)) + \sum_i h'(\nu_i) \Delta(X_i(k))\right].
$$
Recall that $\bar X(k)$ converges in distribution to $\bar X$. Using Skorohod representation and the last inequality in \eqn{eq:bound_discrete_cont}, we can apply Fatou's Lemma to obtain $\limsup_{k\to\infty}$ of the right-hand side of the above display. Thus, we obtain
$$
\limsup_{k\to\infty} \E[F(\bar{X}(k+1))-F(\bar{X}(k))] \le \E\left[- \varepsilon \sum_i h'(\nu_i) g(X_i) + \sum_i h'(\nu_i) \Delta(X_i)\right].
$$
The left-hand side above must be greater or equal to $0$ (because otherwise we would have $\E F(\bar{X}(k)) \to -\infty$).
\end{proof}

In certain particular cases (such as, e.g., when $g(x) = x^\alpha$ with an integer $\alpha$) one can significantly weaken the assumptions on random variables $\xi_i$. We provide the following theorem as an example of this, and also as we will use exactly this choice of functions $g$ and $h$ in section \ref{sec:infinite_single} to prove stability of an infinite network. We note however that we do rely on fluid limits in the proof of the following result.

\begin{theorem} \label{thm:bound_third}
For a discrete-time model defined in Section \ref{sec:model_discrete_finite}, assume that $g(y) = y^2$ and $h(y) = -y^{-1}$. Assume also that $\xi_i$ is a non-negative integer-valued random variable with $\E(\xi_i^3)<\infty$ and $\E(\xi_i) = \lambda_i$ with the vector $\ol$ such that condition \eqn{eq-sub-critical} holds. Then the Markov chain $\{\bar{X}(k), ~k=0,1,\ldots\}$ is stable.

Moreover, fix $\varepsilon > 0$ such that $\lambda_i < \nu_i - \varepsilon$ for all $i$. Then
$$
\varepsilon \sum_i \frac{\E(X_i^2)}{\nu_i^2} \le A \sum_i \frac{\E(X_i)}{\nu_i^2} + B,
$$
where $\bar{X}$ denotes a random element with the stationary distribution of $\bar{X}(\cdotp)$,
$$
A = 3 \sum_i \frac{\E(\xi_i^2) + \lambda_i (1-2 \lambda_i)}{\nu_i^2}
$$
and
$$
B = \sum_i \frac{\E(\xi_i^3) - \lambda_i + 3 \lambda_i^2 - 3 (1-2\lambda_i) (\lambda_i^2 - \lambda_i/2 + \E(\xi_i^2)/2)}{\nu_i^2}.
$$
\end{theorem}

The specific functions $g(\cdot)$ and $h(\cdot)$ are such that the fluid limits 
of the process are well defined. We do use this fact to rely on previous results on stability and existence of moments, as will be seen shortly.

\begin{proof}[Proof of Theorem \ref{thm:bound_third}]

Positive recurrence of the Markov chain $\bar{X}(\cdot)$ under the assumptions (in fact the existence of only the first moments of $\xi_i$'s is sufficient for stability) of the theorem holds due to \cite[Lemma 12]{SnSt2018}, where the stability of the corresponding fluid limits is established.
(Note that, if one assumed convexity of the set $\mathcal C$, stability would follow from earlier results, see, e.g. \cite{Bonald2001,DeVeciana2001}; however, the convexity of the set $\mathcal C$ is not in fact necessary for stability results, which is pointed out in \cite{SnSt2018}). We will consider the stationary version of the process. The finiteness of the third moment of $\xi_i$ (along with stability of fluid limits) guarantees that $\E(X_i^2) < \infty$ (see \cite{Meyn1995}).

Due to stationarity, $\E(X_i(k+1)) = \E(X_i(k))$ and hence
\begin{equation} \label{eq:third_aux_1}
\E(\psi_i) = \lambda_i,
\end{equation}
where for simplicity we write $\psi_i$ instead of $\psi_i(\bar{X})$.

Note that 
\begin{equation} \label{eq:expectations}
\E(X_i^l \eta_i) = \E(\E(X_i^l \eta_i |\bar{X})) = \E(X_i^l \E(\eta_i|\bar{X})) = \E(X_i^l \psi_i)
\end{equation}
for any $l$. Note also that $\eta_i^l = \eta$ a.s. for any $l > 0$. Due to stationarity, we also have $\E(X_i^2(k+1)) = \E(X_i^2(k))$, which is equivalent to
 \begin{align*}
 0 & = \E(\xi_i^2) + \E(\psi_i) - 2\E(\xi_i) \E(\psi_i) + 2\E(\xi_i) \E(X_i) - 2 \E(X_i \psi_i) \\ & = \E(\xi_i^2) + \lambda_i - 2 \lambda_i^2 + 2 \lambda_i \E(X_i) - 2\E(X_i \psi_i),
\end{align*}
where we used \eqref{eq:third_aux_1}, and hence
\begin{equation} \label{eq:third_aux_2}
\E(X_i \psi_i) = \lambda_i \E(X_i) - \lambda_i^2 + \lambda_i/2 - \E(\xi_i^2)/2.
\end{equation}

Assume now that $\E(X_i^3) < \infty$ (we will demonstrate how to drop this additional assumption at the end of the proof). Then the equality of the third moments in stationarity implies
\begin{align}
0 & = \E(\xi_i^3) - \E(\psi_i) + 3\E(\xi_i) \E(\psi_i) - 3 \E(\xi_i^2) \E(\psi_i) + 3 \E(X_i \psi_i) - 3\E(X_i^2 \psi_i) + 3\E(X_i) \E(\xi_i^2) \notag \\ 
& + 3 \E(X_i^2) \E(\xi_i) - 6\E(X_i \psi_i) \E(\xi_i) \notag \\
& = \E(\xi_i^3) - \lambda_i + 3 \lambda_i^2 - 3 \lambda_i \E(\xi_i^2) - 3\E(X_i^2 \psi_i) + 3\E(X_i) \E(\xi_i^2) + 3 \lambda_i \E(X_i^2) \notag \\ 
& +3 (1-2\lambda_i) (\lambda_i \E(X_i) - \lambda_i^2 + \lambda_i/2 - \E(\xi_i^2)/2) \notag \\
& = 3 (\lambda_i \E(X_i^2) - \E(X_i^2 \psi_i)) + A_i \E(X_i) + B_i, \label{eq:third_aux_3}
\end{align}
where we used \eqref{eq:third_aux_1} and \eqref{eq:third_aux_2} and where
$$
A_i = 3 \E(\xi_i^2) + 3\lambda_i (1-2 \lambda_i)
$$
and
$$
B_i = \E(\xi_i^3) - \lambda_i + 3 \lambda_i^2 - 3 (1-2\lambda_i) (\lambda_i^2 - \lambda_i/2 + \E(\xi_i^2)/2).
$$
Due to \eqref{eq:fairness_inequality},
$$
0 \le \sum_i \frac{x_i^2}{\nu_i^2} (\psi_i(\ox) - \nu_i)
$$
for any $\ox$, and hence
\begin{align*}
\sum_i \frac{\lambda_i \E(X_i^2) - \E(X_i^2 \psi_i)}{\nu_i^2} & = \sum_i \frac{(\lambda_i - \nu_i) \E(X_i^2)}{\nu_i^2} + \sum_i \frac{\E((\nu_i - \psi_i) X_i^2)}{\nu_i^2} \\
& \le -\varepsilon \sum_i \frac{\E(X_i^2)}{\nu_i^2}.
\end{align*}
The statement of the Theorem now follows by dividing \eqref{eq:third_aux_3} by $\nu_i^2$ and summing over all $i$. 

We now show that the assumption $\E(X_i^3) < \infty$ can be dropped. Let $M < \infty$ and consider the system with arrivals given by $\xi_i^{(M)} = \min\{\xi_i,M\}$ instead of $\xi_i$. 
Of course, $\E \xi_i^{(M)} \le \E \xi_i$. 
Therefore, the system is stable for each $M$, and let us denote by $\bar{X}^{(M)}$ a random element which has its stationary distribution. 
For each $M$, $\E((X_i^{(M)})^3) < \infty$ (because $\E((\xi_i^{(M)})^4)<\infty$ and \cite{Meyn1995}), and the derivations above imply that
$$
\sum_i \frac{\E((X_i^{(M)})^2)}{\nu_i^2} \le A^{(M)} \sum_i \frac{\E(X_i^{(M)})}{\nu_i^2} + B^{(M)},
$$
with obvious expressions for $A^{(M)}$ and $B^{(M)}$. Since $\E((\xi_i^{(M)})^l) \to \E(\xi_i^l)$ as $M \to \infty$ for $l=1,2,3$, $A^{(M)} \to A$ and $B^{(M)} \to B$ as $M \to \infty$. 

It is easy to check that the sequence $\bar X^{(M)}(\cdot)$ and $\bar X(\cdot)$ satisfy the continuity property.
Indeed, if $\bar X^{(M)}(0) \Rightarrow \bar X(0)$,
we can use Skorohod representation to construct all (random) initial states $\bar X^{(M)}(0)$ and $\bar X(0)$ on a common probability space 
so that $\bar X^{(M)}(0) \to \bar X(0)$ w.p.1. Now, we augment that probability space in the following natural probability space, on which all processes $\bar X^{(M)}(\cdot)$ and $\bar X(\cdot)$
are constructed. The service process is driven by the (independent) sequence of i.i.d. random mappings from a queue length vector $\bar X$ to a service vector
$(\eta_i)$. The arrival process is driven by an (independent) sequence of i.i.d.  arrival vectors $(\xi_i)$.
All processes $\bar X^{(M)}(\cdot)$ and $\bar X(\cdot)$ are then constructed in exactly same, natural way, except in the $M$-th processes
the number of job arrivals is ``clipped'' at $M$, i.e. it is $\xi_i^{(M)} = \min\{\xi_i,M\}$. We directly observe that, w.p.1, 
for any time $k$ and all sufficiently large $M$, $\bar X^{(M)}(k) = \bar X(k)$.

Note that all $X_i^{(M)}$ have uniformly (in $M$) bounded second - and then also first -- moments.
Therefore, we can choose a subsequence of $M$, along which $\bar X^{(M)} \Rightarrow \tilde X$.
We now construct the stationary versions of the processes $\bar X^{(M)}(\cdot)$ (along the subsequence) and a process 
$\bar X(\cdot)$ with $\bar X(0)$ distributed as $\tilde X$, on a common probability space as described above.
Since the sequence $\bar X^{(M)}(\cdot)$ and $\bar X(\cdot)$ satisfy the continuity property,
we see that, w.p.1,
$\bar X^{(M)}(k) \to \bar X(k)$ for any $k$, and thus constructed $\bar X(\cdot)$  
is in fact stationary.
This, in turn, implies that $\bar{X}^{(M)} \Rightarrow \bar{X}$.

 It remains to rewrite the last display as
$$
\sum_i \frac{1}{\nu_i^2} \E [X_i^{(M)} - A^{(M)}/2]^2 \le B^{(M)}+\sum_i \frac{(A^{(M)})^2}{4 \nu_i^2},
$$ 
and apply Fatou's Lemma to obtain
$$
\sum_i \frac{1}{\nu_i^2} \E [X_i - A/2]^2 \le \liminf_{M\to\infty} \sum_i \frac{1}{\nu_i^2} \E [X_i^{(M)} - A^M/2]^2 \le B+\sum_i \frac{A^2}{4 \nu_i^2}.
$$
\end{proof}

\section{Stability analysis of infinite single-hop networks} 
\label{sec:infinite_single}

In this section we provide an application of our moment bounds to establishing stability of infinite networks considered in \cite{Baccelli2018} and \cite{SnSt2018}. The {\em stability of an infinite network} we define as the existence of a proper stationary distribution 
(with all queues finite with probability $1$). 

In our analysis of finite systems in the previous sections, only the average service rates (at a given time, given the system state) were of importance and any dependencies between the departures from different queues were not relevant. When we move to analysis of infinite systems, we still will not require that the departures in each time slot are independent (given the system state), but we do have to specify the departure (service) mechanism to make sure that the processes we consider satisfy the continuity and/or monotonicity properties. In particular, the continuity will be the key property which we need to make limit transitions from finite systems to infinite ones. 

For the motivation of the specific service mechanisms that we consider (in particular related to 
wireless networks), we refer the reader to \cite{Baccelli2018} and \cite{SnSt2018}.


\subsection{Model}
\label{sec-inf-model}

The queues (or nodes) are assumed to be located on a $d$-dimensional lattice, with the service rates given by 
\begin{equation} \label{eq:rates_baccelli}
\psi_i(\ox) = \frac{x_i}{\sum_{j \in \mathbb{Z}^d} a_{j-i} x_j},
\end{equation}
where $a_0 = 1$, $a_i = a_{-i} \ge 0$ for all $i \in \mathbb{Z}^d$ and $L = \sup\{|i|: a_i > 0\} < \infty$.
For each $i$, the nodes $j$ within the finite set $\mathcal{N}_i = \{j ~|~a_{j-i}>0\}$ are called {\em neighbours} of $i$. Note that $i\in \mathcal{N}_i$. 

Recall that the arrivals are driven by the set of independent random variables $\xi_i(k)$, 
which represent the number of arrivals into node $i$ at time $k$. The sets $\{\xi_i(k)\}$ are i.i.d. across $k$; and for each fixed $i$, 
$\xi_i(k)$ are i.i.d. across $k$. As before, we denote by $\xi_i$ the generic $\xi_i(k)$, and assume
$$
\E \xi_i^3 < \infty.
$$
We consider the following two service algorithms for the discrete-time case. Our results apply to both. 
(Again, see \cite{SnSt2018} for the motivation of the algorithms.)
Recall that $X_i(k)$ are the queue lengths at time $k$.

{\em Discrete-time service algorithm 1 (D1).} The algorithm is driven by the set of i.i.d. (across node indices $i$ and times $k$) random variables 
$\nu_i(k)$, distributed uniformly in $[0,1]$. The {\em access priority} of node $i$ at time $k$ is $\tau_i(k) = [-\log \nu_i]/X_i(k)$ -- it is exponentially distributed with mean $1/X_i(k)$. (The smaller the $\tau_i(k)$ the ``higher'' the priority.)
Then, node $i$ transmits in slot $k$, if $X_i(k)>0$ and 
$$
\tau_i(k) < \tau_j(k)/a_{j-i} ~~\mbox{for all}~~ j \in \mathcal{N}_i \setminus i.
$$
Note that the probability of node $i$ transmitting, conditioned on $\bar X(k)$, is exactly $\psi_i(\bar X(k))$, 
as required by \eqn{eq:rates_baccelli}. At the same time, the transmissions of the nodes at time $k$, even conditioned on $\bar X(k)$, are {\em not} independent (except in the degenerate case $\mathcal N_i = i$). In fact, in the case when all $a_i$ are either $1$ or $0$, neighbouring nodes can {\em never} transmit simultaneously. 

{\em Discrete-time service algorithm 2 (D2).} This algorithm is much simpler. 
It is also driven by the set of i.i.d. (across node indices $i$ and times $k$) random variables 
$\nu_i(k)$, distributed uniformly in $[0,1]$. Node $i$ transmits in slot $k$, if $X_i(k)>0$ and 
$$
\nu_i(k) < \psi_i(\bar X(k)).
$$
In other words, conditioned on $\bar X(k)$, the probabilities of nodes transmitting are exactly $\psi_i(\bar X(k))$ (as required by \eqn{eq:rates_baccelli}), and the transmissions are independent.

\subsection{Continuity and monotonicity}
\label{sec-mon-cont}

For the infinite system process, we will use continuity 
(as defined in Section \ref{sec:basic})
and monotonicity properties.

For a continuity property to be well-defined, a topology on the process state space needs to be specified.
A state of the process is a set $\bar X =\{X_i\}$ of the queue lengths, i.e. a function of $i$. On this state space (which is uncountable for an infinite system), we consider the natural topology of component-wise convergence.

We also consider the natural component-wise order relation $\bar X \le \bar X^*$ on the state space. With respect to this partial order,
it is easy to see that the process for the system defined above has the following {\em monotonicity} property: two versions of the process, such that $\bar X^*(0) \le \bar X(0)$, can be coupled (constructed on a common probability space), so that $\bar X^*(k) \le \bar X(k)$ at all times $k\ge 0$. We note that this monotonicity only holds for single-hop systems; it does {\em not} hold for the multi-hop system which we will consider later is Section~\ref{sec:infinite_multi}.

\subsection{Auxiliary system on a finite torus}

For a vector $\bar{N} = (N_1,\ldots, N_d)$, denote by $\mathcal{T} = \{i: i_k = -N_k,\ldots,N_k-1\} \subset \mathbb{Z}^d$ the finite subset of points, "wrapped around" to form a torus. Consider the auxiliary version of our system, defined on the finite 
torus $\mathcal{T}$, with the node neighborhood structure being that of the torus.
Denote by $n = \prod_k (2N_k)$ the number of points in $\mathcal{T}$. We will only consider vectors $\bar{N}$ such that $N_k > l$ for each $k$.

Along the lines of \cite[Lemma 11]{SnSt2018}, we can show that for any  such $\mathcal{T}$, the rates \eqref{eq:rates_baccelli} are in fact utility maximising in a certain set. Indeed, denote
$$
\mathcal{C} = \{\bar{\mu}: \quad \text{there exists} \quad \bar{p} \quad \text{such that} \quad \bar{\mu} \le \bar{\psi}(\bar{p}) \}.
$$
We can prove the following optimality result.
\begin{lemma} \label{lemma:fairness}
The rates \eqref{eq:rates_baccelli} are utility maximising (they satisfy relation \eqref{eq:fainess_def_1}) with the functions $g(y) = y^2$, $h(y) = -y^{-1}$ and the set $\mathcal{C}$.
\end{lemma}

\begin{remark}
Using the standard terminology of $\alpha$-fairness, Lemma \ref{lemma:fairness} states that the rates \eqref{eq:rates_baccelli} are $2$-fair in the set $\mathcal{C}$.
\end{remark}

\begin{proof}[Proof of Lemma \ref{lemma:fairness}]
Indeed, due to the definition of the set $\C$, for any $\bar{\mu} \in \C$,
$$
\sum_i x_i \left(\frac{\mu_i}{x_i}\right)^{-1} \ge \sum_i x_i \left(\frac{\psi_i(\bar{p})}{x_i}\right)^{-1}
$$
for the corresponding vector $\bar{p}$. Hence, it is sufficient to show that
$$
\sum_i x_i \left(\frac{\psi_i(\ox)}{x_i}\right)^{-1} \le \sum_i x_i \left(\frac{\psi_i(\bar{p})}{ x_i}\right)^{-1}
$$
for all vectors $\bar{p}$. Note that the LHS of the above is equal to 
$$\sum_i x_i \sum_{j \in \mathcal{T}} a_{j-i} x_j = \sum_i x_i^2 + \sum_i \sum_{j \in \mathcal{T}, j \neq i} a_{j-i} x_i x_j.$$ 
Consider now
\begin{align*}
\sum_i x_i \left(\frac{p_i}{(\sum_{j \in \mathcal{T}} a_{j-i} p_j) x_i}\right)^{-1} & = \sum_i x_i^2 \left(1+ \sum_{j \in \mathcal{T}, j \neq i} \frac{a_{j-i}p_j}{p_i}\right)
\\ & = \sum_i x_i^2 + \frac{1}{2} \sum_i \sum_{j \in \mathcal{T}, j \neq i} \left(x_i^2 \frac{a_{j-i} p_j}{p_i} + x_j^2 \frac{a_{i-j} p_i}{p_j}\right).
\end{align*}
For any $i$ and $j$,
$$
x_i^2 \frac{a_{j-i} p_j}{p_i} + x_j^2 \frac{a_{i-j} p_i}{p_j} = a_{j-i} \left(x_i^2 \frac{ p_j}{p_i} + x_j^2 \frac{p_i}{p_j}\right) \ge 2a_{j-i}  x_i x_j,
$$
where we used the symmetry of the sequence $a_i$. The equality in the above is possible if and only if $x_i^2 \frac{p_j}{p_i} = x_j^2 \frac{p_i}{p_j}$, which is equivalent to $\frac{p_i}{x_i} = \frac{p_j}{x_j}$. Therefore we obtain
$$
\sum_i x_i \left(\frac{p_i}{(\sum_{j \in \mathcal{T}} a_{j-i} p_j) x_i}\right)^{-1} \ge \sum_i x_i^2 + \sum_i \sum_{j \in \mathcal{T}, j \neq i} a_{j-i} x_i x_j,
$$
and the equality is possible if and only if $\frac{p_i}{x_i} = \frac{p_j}{x_j}$ for all $i$ and $j$. This implies that $\frac{p_i}{x_i}$ has to be a constant for each $i$.
\end{proof}

\begin{lemma} 
\label{rem:symmetric_fairness} 
If $\lambda_i = \lambda$ for each $i$, then the existence of $\on \in \C$ such that $\ol < \on$ is equivalent to the inequality $\lambda < \dfrac{1}{\sum_{j \in \mathcal{T}} a_j} = \dfrac{1}{\sum_{j \in \mathbb{Z}^d} a_j}$ for any vector $\bar{N}$ such that $N_k > L$ for all $k$.
\end{lemma}

\begin{proof}[Proof of Lemma \ref{rem:symmetric_fairness}]
Indeed, if $\lambda < \dfrac{1}{\sum_{j \in \mathcal{T}} a_j}$, we can take $\bar{p} = (1,\ldots,1)$ and $\on = \bar{\psi}(\bar{p})$ - such a vector clearly belongs to $\C$, and it is also clear that $\ol < \on$. In the opposite direction, assume that $\ol < \on$ such that $\on \in \C$ and fix the corresponding vector $\bar{p}$. Then
$$
\frac{1}{\lambda} > \frac{\sum_{j \in \mathcal{T}} a_{j-i} p_j}{p_i}
$$
for each $i \in \mathcal{T}$. If we add up these inequalities over all $i \in \mathcal{T}$, we obtain
\begin{align*}
\frac{n}{\lambda} & > \sum_{i \in \mathcal{T}} \frac{\sum_{j \in \mathcal{T}} a_{j-i} p_j}{p_i} = n + \sum_i \sum_{j \in \mathcal{T}, j \neq i} \frac{a_{j-i} p_j}{p_i}
\\ & = n + \frac{1}{2} \sum_i \sum_{j \in \mathcal{T}, j \neq i} \left(\frac{a_{j-i} p_j}{p_i} + \frac{a_{i-j} p_i}{p_j}\right) 
\\ & = n + \frac{1}{2} \sum_i \sum_{j \in \mathcal{T}, j \neq i} a_{j-i} \left(\frac{p_j}{p_i} + \frac{p_i}{p_j}\right) 
\\ & \ge n + n \sum_{j \in \mathbb{Z}^d, j \neq 0} a_j,
\end{align*}
which concludes the proof (recall that $a_0=1$).
\end{proof}

\subsection{Stability analysis} \label{sec:infinite_stability}

In this section we demonstrate how our results on moment bounds allow us to obtain a stability result for an infinite network, along with a second moment bound for a stationary distribution.

We will say that the arrival rates $\lambda_i$ are {\em periodic}, if the following holds: (a) the values of $\lambda_i$ are given for $i$ within the rectangle $\mathcal{I} = [0, \ldots, C_1-1] \times \ldots \times [0, \ldots, C_d-1]$ with some fixed positive integers $C_1,\ldots, C_d$; (b) for any $i\in  \mathbb{Z}^d$ and any $k=1,\ldots,d$, $\lambda_{i + C_k e_k} = \lambda_i$, where $e_k$ is the $k$-th unit coordinate vector (with $k$-th entry equal to $1$ and all other entries equal to zero). 
Similarly, we define periodicity of any other function of $i$.
We will say that random variables $\xi_i$ are {\em i.i.d. up to periodicity},
if they are all independent, and $\xi_{i + C_k e_k}$ and $\xi_i$ have identical distribution for any $i$ and $k$.

\begin{theorem} \label{thm:periodic} Consider periodic rates $\bar \lambda$.
Assume that $\xi_i$ are i.i.d. up to periodicity, and $\E \xi_i^3 < \infty$ for all $i$.
Assume that there exists a periodic $\bar \nu$ from the set
$$
\mathcal{C} = \{\bar{\mu}: \quad \text{there exists} \quad \bar{p} \quad \text{such that} \quad \bar \mu \le \bar \psi(\bar{p}) \}
$$
such that $\bar \lambda < \bar \nu$. Consider an infinite network with arrival rates $\bar \lambda' \le \bar \lambda$ and assume in addition that the per-slot (random) number of arrivals $\xi'_i$ is dominated by 
$\xi_i$ w.p.1, $\xi'_i \le \xi_i$.
Then this infinite network is stable, and there exists a stationary regime with finite second moments $\E X_i^2$ of the queue lengths.
\end{theorem}

\begin{proof}[Proof of Theorem \ref{thm:periodic}] 
Due to monotonicity of the process, it suffices to prove the theorem for the periodic arrival rates $\bar \lambda$ and arrival process $\bar\xi$.

Consider the auxiliary, finite version of our system on the sets $\mathcal{R}_n = \{i: i_k = - n C_k, \ldots, n C_k-1\}$ "wrapped around" to form a torus. (The node neighbourhood structure being that of the torus.)
Then the conditions of the Theorem, along with Lemma~\ref{lemma:fairness}, imply stability and therefore existence of the (unique) stationary measure for the process on $\mathcal{R}_n$. Lemma~\ref{lemma:fairness} and Theorem \ref{thm:bound_third}, along with the periodicity, imply that
\begin{equation} \label{eq:cor_tightness}
\sum_{i \in \mathcal{I}} \frac{\E\left((X_i^{(n)})^2\right)}{\nu_i^2} \le A_1 \sum_{i \in \mathcal{I}} \frac{\E(X_i^{(n)})}{\nu_i^2} + A_2
\end{equation}
with some constants $A_1$ and $A_2$, where the upper index $n$ is used to indicate the finite system on torus $\mathcal{R}_n$. Note that $A_1$ and $A_2$ do not depend on $n$. This implies a uniform in $n$ and $i \in \mathcal{I}$ second moment bound
\beql{eq-2moment-torus}
\E\left((X^{(n)}_i)^2\right) \le C < \infty.
\end{equation}

Let us view each process $\bar X^{(n)}(\cdot)$ as a process on the entire infinite lattice $\mathbb Z^d$; say, by letting $X_i(\cdot) \equiv 0$ for $i\not\in \mathcal R_n$. (We note that the node neighbourhood structure remains as that of the torus, and so the process is still as that on the torus.) Correspondingly, we will view the (stationary) distributions of $\bar X^{(n)}$ as distributions on the entire infinite lattice $\mathbb Z^d$; we see from \eqn{eq-2moment-torus} that these distributions are tight (as distributions on $\mathbb Z^d$). Then there exists a subsequence of (stationary) processes $\bar X^{(n)}(\cdot)$, along which $\bar X^{(n)}(0) \Rightarrow \bar X^*$, where $\bar X^*$ is some proper random element (with all components being finite w.p.1), and then  $\bar X^{(n)}(k) \Rightarrow \bar X^*$ for each $k$.

It is easy to observe that the sequence of processes $\bar X^{(n)}(\cdot)$ and the process $\bar X(\cdot)$ 
(which is the true infinite system process) satisfy the continuity property (in Section \ref{sec:basic}). This means the subsequence of (stationary) processes $\bar X^{(n)}(\cdot)$ and the process $\bar X(\cdot)$, with $\bar X(0)$ distributed as $\bar X^*$, 
can be coupled in a way such that $\bar X^{(n)}(k) \to \bar X(k)$ w.p.1, for each $k\ge 0$. This means that $\bar X(k)$ is equal 
in distribution to $\bar X^*$ for each $k$, i.e. we constructed a stationary version of $\bar X(\cdot)$.
Since $X_i^{(n)} \Rightarrow X_i^*$, Fatou's lemma and \eqn{eq-2moment-torus} imply that $\E X_i^2 \le C < \infty$.
\end{proof}


Consider now a special case -- a symmetric infinite system, which means that random variables $\xi_i$ are i.i.d. From Theorem~\ref{thm:periodic} we obtain the following. 
\begin{corollary}
\label{cor-conj-proof}
Consider the symmetric system and assume 
\beql{eq-sym-stabil}
\lambda < \dfrac{1}{\sum_{j \in \mathbb{Z}^d} a_j}.
\end{equation}
Assume additionally that $\E \xi_i^3 < \infty$.
Then the system is stable and its lower invariant measure (i.e., the stationary distribution dominated by any other)
is such that 
$
\E X_i^2 < \infty.
$
\end{corollary}

Indeed, by Lemmas \ref{lemma:fairness} and  \ref{rem:symmetric_fairness}, and Theorem~\ref{thm:periodic}, condition
\eqn{eq-sym-stabil} ensures stabiltiy. By Theorem~\ref{thm:periodic}, $\E X_i^2<\infty$ holds for {\em some} stationary distribution,
and therefore it holds for the lower invariant measure as well. 

\begin{remark} \label{rem:periodic}

The simplest case where $\lambda_i$ are not all the same is when we consider a network on the line such that the rates are periodic with a period $2$. Denote the different values by $\lambda_1$ and $\lambda_2$. Theorem \ref{thm:periodic} implies stability as long as there exist $p_1$ and $p_2$ such that
$$
\lambda_1 < \frac{p_1}{p_1+2p_2}
$$
and
$$
\lambda_2 < \frac{p_2}{p_2+2p_1}.
$$
One can see that by taking, for instance, $p_1 = 1$ and $p_2 = \delta > 0$, the values of $\lambda_1$ for which stability holds may be taken arbitrarily close to $1$ (of course at the expense of very low values of $\lambda_2$).
\end{remark}

\begin{remark}
The proof of Theorem \ref{thm:periodic} uses only the utility-maximising properties of the rates and the process continuity properties. It is clear that similar arguments may be used to demonstrate stability of infinite networks in many other cases (see the next section for some examples). 
\end{remark}

\subsection{Other networks}

So far in Section \ref{sec:infinite_single} we have only looked at a particular infinite network, motivated by \cite{Baccelli2018} and \cite{SnSt2018}. In this subsection we demonstrate that the methods developed above are rather general and may be applied to a much wider class of networks. We present some examples below but would like to stress that we think this list is not exhaustive.

For simplicity, we restrict ourselves here to the case when queues are located on a one-dimensional lattice (i.e., 
at the integers). 
One can easily see that the results described below also hold for the more general location and neighbourhood structures described in the rest of Section \ref{sec:infinite_single}.

Consider rates given by
\begin{equation} \label{eq:shannon_rates}
\psi^*_i(\ox) = \log\left(1+\frac{x_i}{x_{i-1}+x_i+x_{i+1}}\right) = \log(1+\psi_i(\ox)),
\end{equation}
which provide a better approximation to the 
wireless channel capacity than the rates considered in Section \ref{sec:infinite_single} so far. 
(Ratio $x_i/(x_{i-1}+x_i+x_{i+1})$ represents Singal-to-Interference-Ratio (SIR).)
Consider 
the auxiliary, finite version of our system on the finite set $\{-n, \ldots, n-1\}$ (with the $2n$ nodes) "wrapped around" to form a torus. (The node neighborhood structure being that of the torus.)
Denote
\begin{equation} \label{eq:shannon_set}
\C^* = \left\{\om: \mu_i \le \log\left(1+\frac{p_i}{\sum_{j \in \N_i} p_j}\right) \quad \text{for some} \quad \bar{p} \in \mathbb{R}_+^{2n} \right\}.
\end{equation}
The following result is immediate.

\begin{lemma} \label{lem:shannon}
It holds that
$$
\bar{\psi^*} \in \argmax_{\om \in \C^*} \sum_i g(x_i) \tilde{h}(\mu_i),
$$
where $g(x) = x^2$ and
$$
\tilde{h}(y) = - \frac{1}{e^y-1}.
$$
\end{lemma}

Indeed, the statement of the lemma reads
$$
\sum_i g(x_i) \tilde{h}(\mu_i) \le \sum_i g(x_i) \tilde{h}(\psi^*_i)
$$
for all $\om \in \C^*$. This is equivalent to
$$
\sum_i g(x_i) \tilde{h}\left(\log\left(1+\frac{p_i}{p_{i-1}+p_i+p_{i+1}}\right) \right) \le \sum_i g(x_i) \tilde{h}(\tilde{\psi^*_i})
$$
for all $\bar{p} \in \mathbb{R}_+^{2n}$. Taking into account definitions of $g$, $\tilde{h}$ and $\tilde{\psi_i}$, the above is equivalent to
$$
- \sum_i x_i^2 \left(\frac{p_i}{p_{i-1}+p_i+p_{i+1}}\right)^{-1} \le - \sum_i x_i^2 \psi_i^{-1},
$$
which follows immediately from Lemma \ref{lemma:fairness}.

As functions $g$ and $\tilde{h}$ in the statement of Lemma \ref{lem:shannon} satisfy all the conditions imposed in Section \ref{sec:finite}, we obtain {\em stability for any finite network,} as long as $\ol$ is within $\C^*$. Moreover, one can also readily see that fluid limits for such a system are well defined and stable. From here we obtain that the result of Theorem \ref{thm:bound_third} holds, with the same {\em finite-system second-moment bound}, except with 
$1/\nu_i^2$ replaced by $\tilde{h}'_i(\nu_i)$. Then the strategy of Section \ref{sec:infinite_stability} can thus be followed to obtain {\em stability of the infinite version of the network with rates $\bar{\psi}^*$.}

More generally, assume that
$$
\sum g(x_i) h(\psi_i(\ox)) \ge \sum g(x_i) h(\mu_i)
$$
for all $\om \in \C$. We can rewrite this as
$$
\sum g(x_i) \tilde{h}(f(\psi_i(\ox))) \ge \sum g(x_i) \tilde{h}(f(\mu_i))
$$
for all $\om \in \C$, where $\tilde{h} = h \circ f^{-1}$ and $f^{-1}$ is the inverse of $f$ (which we assume to be increasing). The above may be rewritten again as
$$
\sum g(x_i) \tilde{h}(f(\psi_i(\ox))) \ge \sum g(x_i) \tilde{h}(\nu_i)
$$
for all $\on \in \C^*$, where
$$
\C^* = f(\C) = \{\on: (f^{-1}(\nu_1), \ldots, f^{-1}(\nu_N)) \in \C\}.
$$
Therefore, as long as $\tilde{h}$ is concave (which in general does not necessarily hold), then all conditions of Section \ref{sec:finite} are satisfied, and we obtain stability of a finite network with rates $f(\psi_i)$, provided $\ol$ is within $\C^*$. If, further, fluid limits of finite networks are well-defined and stable, additional steady-state moment bounds can be obtained, which in turn 
leads to stability of infinite networks.

It is important to note that, {\em if one is only interested in the stability of finite networks} with the rates considered so far in Section \ref{sec:infinite_single} (including the rates considered in this subsection), this may be established {\em without} the use of utility maximisation. In fact, for finite networks, one can demonstrate stability for a wider still class of rates, including rates of the form
$$
\log\left(1+\frac{x_i}{x_{i-1}+x_i+x_{i+1}+B}\right).
$$
(The proof is a slight generalization of that of Lemma 9 in \cite{SnSt2018}.)
Here $B>0$ represents the ``background noise'' and $x_i/(x_{i-1}+x_i+x_{i+1}+B)$ is Singal-to-Interference-plus-Noise-Ratio (SINR).
These rates provide an even more realistic approximation of the 
wireless channel capacity. We are not aware of any utility-maximisation properties satisfied by these rates, yet (a slight generalization of)  \cite[Lemma 9]{SnSt2018} provides stability of a finite network with these transmission rates, as long as the vector of arrival intensities belongs to the set $\C^*$ defined in \eqref{eq:shannon_set}. This is thanks to the fact that fluid limits for the system with these rates are ``same'' 
(satisfy same properties)
as for the system with rates \eqref{eq:shannon_rates}.

We emphasize again that it is the utility-maximisation property, enjoyed by rates \eqref{eq:shannon_rates}, that allows us to obtain 
not only the stability, but also steady-state 
moment bounds, for finite networks. The finite network moment bounds, in turn, allow us to apply the methods developed in this paper to construct a stationary distribution for an infinite network, and moreover obtain moment bounds for this stationary distribution.

\section{Stability analysis of an infinite multi-hop network} 
\label{sec:infinite_multi}

In this section we demonstrate how techniques similar to the ones we used in the single-hop case may be used to demonstrate the existence of invariant measures for infinite multi-hop networks, where, upon a service completion at a given queue, a job may leave the system or enter the queue of a neighbouring node. 

Multi-hop networks have an additional layer of difficulty as the movement of jobs between different queues complicates the dependence structure of the queue states further. Multi-hop networks are notoriously difficult to analyse, and we consider significantly stronger assumptions on the structure of the network, with strictly i.i.d. arrival processes and with symmetric routing (see \cite{SnSt2018}). 

Specifically, the model is as follows.
Just as in Section~\ref{sec:infinite_single}, the nodes (queues) are located on the $d$-dimensional lattice. 
The exogenous arrival processes are strictly i.i.d.; namely, the random variables $\xi_i$ representing the numbers of new arrivals are i.i.d. with $\E(\xi_i) = \lambda q$ with a fixed $q \in (0,1)$ and with $\E(\xi_i^2) < \infty$. 
The service is governed by either algorithm (D1) or (D2), specified in Section~\ref{sec:infinite_single}.
Upon a service completion at any node, a job leaves the system with probability $q$ or joins the queue of a neighbour of node $i$ (i.e., connected by one lattice edge),
chosen independently at random (i.e., each neighbour is chosen with probability $(1-q)/2^d)$). All the routing decisions are taken independently of everything else. The service rates are given by
$$
\psi_i(\ox) = \frac{x_i}{\sum_{j \in \mathcal{N}_i} x_j},
$$
where $\mathcal{N}_i$ is the neighbourhood of node $i$ on the $\mathbb{Z}^d$ lattice, which by convention includes node $i$ itself.
(In other words, the service rates are a special case of those considered in Section~\ref{sec:infinite_single}, with the neighbourhood
$\mathcal N_i$ of node $i$ including specifically the neighbours in terms of the lattice, and with $a_{j-i}=1$ for all $j\in \mathcal N_i$.)

Consider the auxiliary, finite version of our system on the set $\mathcal{T}_n = \{i: i_k=-n,\ldots,n-1\}$, "wrapped around" to form a torus. (The node neighbourhood structure is that of the torus.) Then $\mathcal{T}_n$ is a finite $2^d$-regular graph, and \cite[Theorem 5]{SnSt2018} implies that if $\lambda < 1/(2^d+1)$, then the system is stable. Therefore, there exists a stationary distribution of the number of messages in each queue. Due to symmetry, the (stationary) numbers of messages in any two queues are identically distributed, and we will consider queue $0$ for simplicity. Denote the stationary number of messages in queue $0$ in the system on torus $\mathcal{T}_n$ by $X^{(n)}$.

We want to emphasise that the described multi-hop process 
(for both the infinite system and a finite torus)
 is {\em not monotone} (unlike in the single-hop model of Section~\ref{sec:infinite_single}), and this is in fact one of the key challenges of the multi-hop system analysis. 
Versions of this process, however, do have continuity properties, which we will exploit, just as in the single-hop case. 

\begin{theorem} \label{thm:multi_bound}
Consider the multi-hop model on torus $\mathcal{T}_n$, described above, and denote by $\xi$ a random variable with the distribution of $\xi_i(k)$ for any $i$ and $k$. Assume that $\E(\xi^2) < \infty$ and $\E(\xi) = \lambda < \frac{1}{2^d+1}$. Then
\begin{equation} \label{eq:multi_bound}
\E(X^{(n)}) \le \frac{\E(\xi^2) + 2^d(1-q) \lambda + \lambda - 2 \lambda^2 q^2}{2q\left(\frac{1}{2^d+1} - \lambda\right)}.
\end{equation}
\end{theorem}

\begin{proof}[Proof of Theorem \ref{thm:multi_bound}]

Consider the stationary version $\bar X^{(n)}(\cdot)$ of the Markov chain.
For ease of notation, in this proof we will drop the superscript $(n)$, and write 
$\bar X(\cdot)$ instead of $\bar X^{(n)}(\cdot)$. Similarly, we will write $X$ instead of $X^{(n)}$ for a random vector distributed as $X^{(n)}_i(k)$
(for any $i$ and $k$, by stationarity and symmetry). We can describe the evolution of $X_i(k)$ as
\begin{equation} \label{eq:multi_evolution}
X_i(k+1) = X_i(k) + \xi_i(k) + \sum_{j \in \N_i, j \neq i} I_{ji}(k) \eta_j(k) - \eta_i(k),
\end{equation}
where random variables $I_{ji}(k)$ are indicator functions of events that a message potentially leaving node $j$ in time slot $k$ will choose node $i$ as its destination. For ease of notation, as we only consider a single time slot in what follows, we are going to simply write $\xi_i$, $\eta_l$ and $I_{ji}$.

As in the previous sections, note that random variables $\eta_l$ can only take values $0$ and $1$ and
$$
\E(\eta_l|\bar{X}) = \P(\eta_l=1|\bar{X}) = \psi_l(\bar{X}) \quad \text{a.s.}
$$
Note also that
$$
\E(I_{ji}) = (1-q) \frac{1}{2^d}
$$
for all $j$ and $i$. Due to stationarity of the process $\bar{X}(\cdot)$, $X_i(k)$ and $X_i(k+1)$ have the same distributions.

Moreover, this distribution has a finite mean: 
\beql{eq-finite-mean}
\E(X_i(k+1)) = \E(X_i(k)) = \E X < \infty.
\eeql
Indeed, $\E \xi^2 < \infty$. We also know from \cite{SnSt2018} that the system {\em fluid limits are stable}. (For our system, the definition of fluid limits -- called fluid sample paths in \cite{SnSt2018} -- as well as that of fluid limit stability for a finite system, are given in Section 4 of \cite{SnSt2018}. The fluid limit stability proof is in Section 6 of \cite{SnSt2018}.) By results of \cite{Meyn1995}, the combination of the finite second moment of the arrival process $\E \xi^2 < \infty$,  
finite second (in fact, any positive) moment of the number of departures from any node at any time,
and fluid limit stability, implies that $\E X < \infty$.

From \eqref{eq:multi_evolution} and \eqn{eq-finite-mean},
\begin{equation} \label{eq:multi_Lyapunov1}
0 = \lambda q + \sum_{j \in \N_i, j\neq i} (1-q) \frac{1}{2^d} \E(\eta_{j}) - \E(\eta_i),
\end{equation}
where we used the fact that $I_{ji}$ and $\eta_j$ are independent. Note now that
$$
\E(\eta_l) = \E(\E(\eta|\bar{X})) = \E(\psi_l(\bar{X})) = \E\left(\frac{X_l}{\sum_{j \in \N_l} X_j}\right)
$$
for any $l$, and, due to the symmetry of the model, it does not depend on $l$. Hence, continuing \eqref{eq:multi_Lyapunov1},
$$
0 = \lambda q + 2^d (1-q) \frac{1}{2^d} \E(\eta_i) - \E(\eta_i),
$$
implying
\begin{equation} \label{eq:Eeta}
\E(\eta_i) = \lambda
\end{equation}
for any $i$. Denote $A_i = \sum_{j \in \N_i, j \neq i} I_{ji} \eta_j$.

Assume first that $\E X^2 < \infty$. (We will show later in the proof how to get rid of this additional assumption.)
Stationarity of the process $\bar{X}(\cdot)$ implies that $\E(X^2_i(k+1)) = \E(X^2_i(k))$ and hence, from \eqref{eq:multi_evolution},
\begin{align}
0 & = \E(\xi_i^2) + \E(A_i^2) + \E(\eta_i^2) - 2\E(A_i \eta_i) + 2\E(\xi_i (A_i-\eta_i)) +2\E(X_i \xi_i) + 2\E(X_i (A_i - \eta_i)) \notag
\\ & \le \E(\xi_i^2) + \E(A_i^2) + \lambda + 2 \lambda q \left(\E(A_i) - \E(\eta_i)\right) + 2 \lambda q \E(X_i) + 2\E(X_i (A_i - \eta_i)) \notag
\\ & = \E(\xi_i^2) + \E(A_i^2) + \lambda - 2 \lambda^2 q^2  + 2 \lambda q \E(X_i) + 2\E(X_i (A_i - \eta_i)). \label{eq:multi_Lyapunov2}
\end{align}
In the derivations above we used the independence of $\xi_i$ from all other random variables, the fact that $\E(\eta_i^2) = \E(\eta_i) = \P(\eta_i=1)$, equation \eqref{eq:Eeta} and finally, in the last equality, a simple calculation of $\E(A_i)$ already performed earlier in this proof (see \eqref{eq:multi_Lyapunov1}).

We consider some of the terms above separately. First,
\begin{equation} \label{eq:multi_proof_aux_2}
\E(A_i^2) = \E\left(\left(\sum_{j \in \N_i, j \neq i} I_{ji} \eta_j\right)^2 \right) \le 2^d \E \left(\sum I_{ji}^2 \eta_j^2  \right) = 2^d \E \left(\sum I_{ji} \eta_j  \right) = 2^d (1-q) \lambda
\end{equation}
where we used convexity of the function $x^2$, independence of $I$'s and $\eta$'s, as well as the facts that all the random variables concerned only take values $0$ and $1$ and  therefore are equal to their squares.

Let us note now that
$$
\E(X_i \eta_j) = \E(\E(X_i \eta_j|\bar{X})) = \E(X_i \E(\eta_j|\bar{X})) = \E(X_i \psi_j) = \E\left(X_i \frac{X_j}{\sum_{l \in \N_j} X_l}\right)
$$
for any $i$ and $j$. It is clear that, due to the symmetry of the model, for any $j \in \N_i$, the pairs $(X_i, \eta_j)$ and $(X_j, \eta_i)$ have identical distributions which implies, in particular, that
$$
\E(X_i \eta_j) = \E(X_j \eta_i).
$$
Consider now
\begin{align}
& \E(X_i(A_i - \eta_i)) = \sum_{j \in \N_i, j \neq i} \E(X_i I_{ji} \eta_j) - \E(X_i \eta_i) = (1-q)\frac{1}{2^d} \sum_{j \in \N_i, j \neq i} \E(X_i \eta_j) - \E(X_i \eta_i) \notag
\\ & = (1-q)\frac{1}{2^d} \sum_{j \in \N_i, j \neq i} \E(X_j \eta_i) - \E(X_i \eta_i)   = (1-q)\frac{1}{2^d} \sum_{j \in \N_i, j \neq i} \E(X_j \psi_i) - \E(X_i \psi_i) \notag
\\ & = \E\left(\psi_i \left((1-q)\frac{1}{2^d} \sum_{j \in \N_i, j \neq i} X_j - X_i \right) \right) \notag
\\ & = \E\left(\frac{X_i}{\sum_{j \in \N_i} X_j}\left((1-q)\frac{1}{2^d} \sum_{j \in \N_i, j \neq i} X_j + (1-q)\frac{1}{2^d} X_i - \left(1+(1-q)\frac{1}{2^d}\right) X_i\right)\right) \notag
\\ & = (1-q) \frac{1}{2^d} \E(X_i) - \left(\frac{2^d+1}{2^d} -q\frac{1}{2^d}\right) \E(X_i \psi_i). \label{eq:multi_proof_aux_3}
\end{align}
Convexity of the function $1/x$ implies that
$$
\frac{1}{\sum_i X_i}\sum_i X_i \psi_i = \sum_i \frac{X_i}{\sum_i X_i}\frac{1}{(\sum_{j \in \N_i} X_j)/X_i} \ge \frac{\sum_i X_i}{\sum_i \sum_{j \in \N_i} X_j} = \frac{1}{2^d+1},
$$
for any $\bar{X}$. Above the summation is taken over all nodes $i$ in our graph. We can rewrite the above as
$$
\sum_i X_i \psi_i \ge \frac{1}{2^d+1} \sum_i X_i
$$
and take the expectations on both sides of the equation. Noting that, due to the symmetry of the model, $\E(X_i \psi_i)$ does not depend on $i$, we conclude that
$$
\E(X_i \psi_i) \ge \frac{1}{2^d+1} \E(X_i)
$$
and, plugging this into \eqref{eq:multi_proof_aux_3}, we obtain
\begin{equation} \label{eq:multi_proof_aux_4}
\E(X_i(A_i - \eta_i)) \le -q \frac{1}{2^d+1} \E(X_i).
\end{equation}
We can now plug \eqref{eq:multi_proof_aux_2} and \eqref{eq:multi_proof_aux_4} into \eqref{eq:multi_Lyapunov2} to obtain
$$
0 \le \E(\xi_i^2) + 2^d (1-q) \lambda + \lambda - 2 \lambda^2 q^2 + 2q\E(X_i) \left(\lambda - \frac{1}{2^d+1}\right),
$$
which implies the statement of Theorem \ref{thm:multi_bound}.

We now show how to remove the additional assumption $\E X^2 < \infty$. Consider the system with truncated arrival quantities
$\xi^{(M)} = \max\{\xi,M\}$. The corresponding process is stable for any $M>0$, and we denote by $\bar X^{(M)}(\cdot)$ its stationary version,
and by $X^{(M)}$ a generic $X_i^{(M)}(k)$. For each $M$, $\E [\xi^{(M)}]^3 < \infty$ (in fact, of course, $\E [\xi^{(M)}]^m < \infty$ for any $m\ge 0$). As already described above in this proof, it is proved in \cite{SnSt2018} that the system fluid limits are stable. 
Using again the results of \cite{Meyn1995}, the combination of the finite third moment of the arrival process $\E [\xi^{(M)}]^3 < \infty$,  
finite third (in fact, any positive) moment of the number of departures from any node at any time,
and fluid limit stability, implies that $\E [X^{(M)}]^2 < \infty$ for any $M$. (In fact, $\E [X^{(M)}]^m < \infty$ for any $M$ and any $m\ge 0$). 

Therefore, for $\E X^{(M)}$, we have the upper bound \eqn{eq:multi_bound} with $\lambda$ replaced by $\E \xi^{(M)}$,
and $\E \xi^2$ replaced by $\E [\xi^{(M)}]^2$. Choose a sequence $M\uparrow\infty$. Then, $\E X^{(M)}$ is uniformly upper bounded along this sequence, and therefore the sequence of distributions of $X^{(M)}$ is tight.
We further observe that the sequence of processes $\bar X^{(M)}(\cdot)$ and the process $\bar X(\cdot)$ satisfy the continuity property
(defined in Section \ref{sec:basic}). Proceeding analogously to the argument we used at the end of the proof of Theorem~\ref{thm:periodic},
we obtain that $X^{(M)} \Rightarrow X$. 
By Fatou's lemma, $\E X \le \liminf_{M\to\infty} \E X^{(M)}$, and the $\liminf$ is upper bounded by the RHS of \eqn{eq:multi_bound}.
\end{proof}

The fact that the bound in \eqref{eq:multi_bound} does not depend on $n$ allows us to prove a result on an infinite-lattice multi-hop model.

\begin{theorem} 
\label{multi_infinity}
Consider the multi-hop model of this section defined for the entire lattice $\mathbb{Z}^d$ and assume that $\lambda < 1/(2^d+1)$. Then the process is stable. Moreover, there is a translation-invariant stationary distribution, for which
\begin{equation} 
\label{eq:multi_bound-inf}
\E X \le \frac{\E(\xi^2) + 2^d(1-q) \lambda + \lambda - 2 \lambda^2 q^2}{2q\left(\frac{1}{2^d+1} - \lambda\right)},
\end{equation}
where $X$ has the distribution of $X_i(k)$ (for any $i$ and $k$) in steady-state.
\end{theorem}

\begin{remark}
Since the multi-hop process is not monotone, the constructions similar to that of \cite{Baccelli2018} cannot be applied. We provide a different construction, based on continuity alone.  Note that Theorem~\ref{multi_infinity} does {\em not} claim any form of the stationary distribution uniqueness.
The uniqueness properties (among the stationary distributions with finite second moments of the queue lengths) derived in 
\cite{Baccelli2018} and in this paper for the single-hop models, relied in essential way on the process monotonicity.
\end{remark}

\begin{proof}[Proof of Theorem \ref{multi_infinity}] We already know that for each torus $\mathcal{T}_n$ there exists a (unique) stationary 
distribution of the corresponding process $\bar X^{(n)}(\cdot)$. (It is translation-invariant, of course, by symmetry.)
We can view this distribution as the distribution on the entire lattice
$\mathbb Z^d$. Moreover, the uniform in $n$ bound \eqn{eq:multi_bound} on the expected queue length implies that these distributions (viewed as distributions on the entire lattice) are tight.
It is easy to see that the sequence of processes $\bar X^{(n)}(\cdot)$ and the process $\bar X(\cdot)$ (i.e., the ``true'' infinite-lattice process) 
satisfy the continuity property (as in Section \ref{sec:basic}). Proceeding analogously to the argument we used in the last two paragraphs of the proof of Theorem~\ref{thm:periodic}, we can construct a proper stationary process $\bar X(\cdot)$ for the infinite system. The constructed stationary distribution of $\bar X(\cdot)$ is a limit of those of $\bar X^{(n)}(\cdot)$, and therefore translation-invariant.
Finally, \eqn{eq:multi_bound-inf} follows from \eqn{eq:multi_bound} and Fatou's lemma.
\end{proof}

\section{Continuous-time setting} \label{sec:continuous}

In this Section we describe how our results translate to the continuous-time Markovian setting. The proofs follow the same lines as those provided in Sections \ref{sec:finite}--\ref{sec:infinite_multi} for the discrete-time setting, with minor changes. In the case of finite networks, proofs of stability and moment bounds in continuous time (results corresponding to those in Section \ref{sec:finite}) turn out to be simplified versions of the proofs in discrete time, as, due to Markovian assumptions, the probability that two or more events happen in a small time interval is negligible. We thus restrict ourselves to short descriptions of proofs, pointing out where they start repeating discrete-time proofs.

As far as construction of stationary distributions for infinite networks is concerned, proofs in continuous time are based on the same continuity and monotonicity properties as their discrete-time analogues and thus repeat them almost verbatim. For this reason,
for infinite networks, we
 only present the results (without proofs), corresponding to those in Sections \ref{sec:infinite_single} and \ref{sec:infinite_multi}.

\subsection{Finite networks: model.} \label{sec:model_cont_finite}

Assume, as before, that there are $N$ interacting queues. Arrivals into queue $i$ occur according to a Poisson process with a constant rate $\lambda_i$, independent of all the other processes. The instantaneous departure rate from queue $i$ at time $t$,
conditioned on the state $\bar X(t)$, is $\psi_i(\bar X(t))$; more precisely, the number of departures up to time $t$ is 
$\Pi_i(\int_0^t \psi_i(\bar X(\tau))d\tau)$, where $\Pi_i(\cdot)$ are independent unit-rate Poisson processes.

We assume that all the conditions on the functions $\psi_i(\cdotp), ~ i=1,\ldots,N$ and on the arrival intensities $\ol$ imposed in Section \ref{sec:model_discrete_finite}, hold. More precisely, we assume that the functions $\psi_i(\cdotp),~i=1,\ldots,N$ satisfy condition \eqref{eq:fainess_def_1} with functions $h$ and $g$ satisfying Conditions {\it (H)} and {\it (G)}, respectively; and we assume that the arrival intensities $\ol$ satisfy condition \eqref{eq-sub-critical}. We will also use functions $G$ and $F$ defined in \eqref{eq:def_G} and \eqref{eq:def_F}, respectively.

\subsection{Finite networks: stability analysis.}

\begin{theorem} \label{thm:stability_cont}
For the continuous-time model defined in Section~\ref{sec:model_cont_finite} such that condition \eqref{eq-sub-critical} holds, the Markov process $\{\bar{X}(t)\}_{t \ge 0}$ is stable.
\end{theorem}

\begin{proof}[Proof of Theorem \ref{thm:stability_cont}]

The proof is a simplified version of that of Theorem \ref{thm:fairness_stability}. By the Lyapunov-Foster criterion, in order to show positive recurrence, it is sufficient to show that
$$
\sum_i (\lambda_i (F(x_i+1)-F(x_i) + \psi_i(\ox)(F(x_i-1)-F(x_i)) < -\delta,
$$
for the function $F$ defined in \eqref{eq:def_F}, some $\delta > 0$, and for values of $\ox$ outside of a compact set. This may be found in, e.g. \cite{Tweedie1981}. 

Note that the expression on the LHS of the above may be written as
$$
\sum_i h'(\nu_i) (\lambda_i (G(x_i+1) - G(x_i)) + \psi_i(\ox)(G(x_i-1)-G(x_i)),
$$
which is equal to \eqref{eq:discrete_cont}, and the rest of the proof of Theorem \ref{thm:fairness_stability} applies.
\end{proof}

\subsection{Finite networks: moment bounds.}

As in discrete time, once stability is established, one can use similar arguments to establish steady-state moment bounds.

Since the stationary regime exists under conditions of Theorem \ref{thm:stability_cont}, in this section we will write $\bar X$ to represent a random vector with the distribution equal to that of $\bar X(t)$ in the stationary regime.

\begin{theorem} \label{thm:fair_tight_cont}
Assume that all conditions of Theorem \ref{thm:stability_cont} hold. Fix $\varepsilon > 0$ such that $\lambda_i < \nu_i - \varepsilon$ for all $i$. Then
$$
\sum_i h'(\nu_i) \E(g(X_i)) \le \frac{1}{\varepsilon} \sum_i h'(\nu_i) \E(\Delta(X_i)).
$$

\end{theorem}

\begin{proof}[Proof of Theorem \ref{thm:fair_tight_cont}] 
Note that condition \eqref{eq:cond_Delta_cor} implies that $g(x) = o(e^{ax})$ as $x \to \infty$ for any $a > 0$. This, in turn, implies that $G(x) = o(e^{ax})$ as $x \to \infty$ for any $a > 0$. Note also that, as arrival flows are given by Poisson processes, for any $t$ and $\ox$, there exists $a>0$ such that 
$$
\E(e^{a\bar{X}(t)} ~|~ \bar{X}(0) = \ox) < \infty.
$$ 
Then $\E(F(\bar{X}(t))|\bar{X}(0) = \ox) < \infty$  for any $t$ and any $\ox$. 

Fix $\varepsilon > 0$ such that $\lambda_i < \nu_i - \varepsilon$ for all $i$. Standard Poisson-process arguments imply that
\begin{align*}
& \frac{d(\E(F(\bar{X}(t))~|~\bar{X}(0) =\ox))}{dt} \\ & = \sum_i h'(\nu_i) (\lambda_i \E(G(X_i+1) - G(X_i)) + \E(\psi_i(\ox)(G(x_i-1)-G(x_i))),
\end{align*}
and the RHS of the above is equal to \eqref{eq:bound_discrete_cont}. The rest of the proof is the same as that of Theorem~\ref{thm:fair_tight}, with $\E \bigl(F(\bar{X}(k+1)) - F(\bar{X}(k)) ~|~\bar{X}(k) = \bar{x}\bigr)$ replaced with $\dfrac{d(\E(F(\bar{X}(t))~|~\bar{X}(0) =\ox))}{dt}$.
\end{proof}

\subsection{Stability analysis of infinite single-hop networks}

We now turn our attention to infinite networks. In the single-hop setting, as in the discrete case, the queues (or nodes) are assumed to be located on a $d$-dimensional lattice, with the service rates given by 
\begin{equation} \label{eq:rates_baccelli_cont}
\psi_i(\ox) = \frac{x_i}{\sum_{j \in \mathbb{Z}^d} a_{j-i} x_j},
\end{equation}
where $a_0 = 1$, $a_i = a_{-i} \ge 0$ for all $i \in \mathbb{Z}^d$ and $L = \sup\{|i|: a_i > 0\} < \infty$.
For each $i$, the nodes $j$ within the finite set $\mathcal{N}_i = \{j ~|~a_{j-i}>0\}$ are called {\em neighbours} of $i$. Note that $i\in \mathcal{N}_i$. 

Unlike the discrete-time case, these assumptions describe the system completely and no additional structural constructions are needed in the continuous-time case.

We say that the arrival rates $\lambda_i$ are {\em periodic}, if the following holds: (a) the values of $\lambda_i$ are given for $i$ within the rectangle $\mathcal{I} = [0, \ldots, C_1-1] \times \ldots \times [0, \ldots, C_d-1]$ with some fixed positive integers $C_1,\ldots, C_d$; (b) for any $i\in  \mathbb{Z}^d$ and any $k=1,\ldots,d$, $\lambda_{i + C_k e_k} = \lambda_i$, where $e_k$ is the $k$-th unit coordinate vector (with $k$-th entry equal to $1$ and all other entries equal to zero). 

The following result is a continuous-time analogue of \ref{thm:periodic}.

\begin{theorem} \label{thm:periodic_cont} Consider periodic rates $\bar \lambda$.
Assume that there exists a periodic $\bar \nu$ from the set
$$
\mathcal{C} = \{\bar{\mu}: \quad \text{there exists} \quad \bar{p} \quad \text{such that} \quad \bar \mu \le \bar \psi(\bar{p}) \}
$$
such that $\bar \lambda < \bar \nu$. Consider an infinite network with arrival rates $\bar \lambda' \le \bar \lambda$.
Then this infinite network is stable, and there exists a stationary regime with finite second moments $\E X_i^2$ of the queue lengths.
\end{theorem}

As in the discrete-time case, we immediately obtain the following corollary.

\begin{corollary}
\label{cor-conj-proof_cont}
Consider the symmetric system where $\lambda_i = \lambda$ for all $i$ and assume 
\beql{eq-sym-stabil_cont}
\lambda < \dfrac{1}{\sum_{j \in \mathbb{Z}^d} a_j}.
\end{equation}
Then the system is stable and its lower invariant measure (i.e., the stationary distribution dominated by any other)
is such that 
$
\E X_i^2 < \infty.
$
\end{corollary}

In the continuous-time setting the corollary above proves Conjecture 1.12 in \cite{Baccelli2018}, with all its implications stated in \cite[Section 1.1]{Baccelli2018}, in particular the uniqueness of the stationary regime 
{\em with finite second moments of the queue lengths}, for the infinite symmetric network.

We note also that our results allow us to consider networks which are not necessarily symmetric, and remark \ref{rem:periodic} applies to the continuous-time setting too, significantly expanding the range of arrival intensities for which stability of infinite networks holds.

%

\subsection{Stability analysis of infinite multi-hop networks}

We now consider infinite multi-hop networks. Just as in the single-hop setting, the nodes (queues) are located on the $d$-dimensional lattice. 

The arrival processes are such that $\lambda_i = \lambda q$ with a fixed $q \in (0,1)$, for all $i$.

Upon a service completion at any node, a job leaves the system with probability $q$ or joins the queue of a neighbour of node $i$ (i.e., connected by one lattice edge),
chosen independently at random (i.e., each neighbour is chosen with probability $(1-q)/2^d)$). All the routing decisions are taken independently of everything else. The service rates are given by
$$
\psi_i(\ox) = \frac{x_i}{\sum_{j \in \mathcal{N}_i} x_j},
$$
where $\mathcal{N}_i$ is the neighbourhood of node $i$ on the $\mathbb{Z}^d$ lattice, which by convention includes node $i$ itself.

Consider the auxiliary, finite version of our system on the set $\mathcal{T}_n = \{i: i_k=-n,\ldots,n-1\}$, "wrapped around" to form a torus. (The node neighbourhood structure is that of the torus.) Then $\mathcal{T}_n$ is a finite $2^d$-regular graph, and \cite[Theorem 5]{SnSt2018} implies that if $\lambda < 1/(2^d+1)$, then the system is stable. Therefore, there exists a stationary distribution of the number of messages in each queue. Due to symmetry, the (stationary) numbers of messages in any two queues are identically distributed, and we will consider queue $0$ for simplicity. Denote the stationary number of messages in queue $0$ in the system on torus $\mathcal{T}_n$ by $X^{(n)}$.

\begin{theorem} \label{thm:multi_bound_cont}
Consider the multi-hop model on torus $\mathcal{T}_n$, described above. Assume that $\lambda < \frac{1}{2^d+1}$. Then
\begin{equation} \label{eq:multi_bound_cont}
\E(X^{(n)}) \le \frac{\lambda}{q\left(\frac{1}{2^d+1} - \lambda\right)}.
\end{equation}
\end{theorem}

The fact that the bound in \eqref{eq:multi_bound_cont} does not depend on $n$ allows us to prove a result on an infinite-lattice multi-hop model.

\begin{theorem} 
\label{multi_infinity_cont}
Consider the multi-hop model of this section defined for the entire lattice $\mathbb{Z}^d$ and assume that $\lambda < 1/(2^d+1)$. Then the process is stable. Moreover, there is a translation-invariant stationary distribution, for which
\begin{equation} 
\label{eq:multi_bound-inf_cont}
\E X \le \frac{\lambda}{q\left(\frac{1}{2^d+1} - \lambda\right)},
\end{equation}
where $X$ has the distribution of $X_i(k)$ (for any $i$ and $k$) in steady-state.
\end{theorem}

\end{document}